\documentclass[twoside]{article}

\title{Local Martingale and Pathwise Solutions for an Abstract Fluids Model}
\author{Arnaud Debussche$^{\sharp}$, Nathan Glatt-Holtz$^\flat$ and Roger Temam$^\flat$}
\date{}

\usepackage{amsfonts, amssymb, amsmath, amsthm}
\usepackage[colorlinks=true, pdfstartview=FitV, linkcolor=blue, 
            citecolor=blue, urlcolor=blue]{hyperref}
\usepackage[usenames]{color}
\definecolor{Red}{rgb}{0.7,0,0.1}
\definecolor{Green}{rgb}{0,0.7,0}
\usepackage{accents}
\usepackage{comment}

\numberwithin{equation}{section}

\newtheorem{Thm}{Theorem}[section]
\newtheorem{Lem}{Lemma}[section]
\newtheorem{Prop}{Proposition}[section]

\newtheorem{Def}{Definition}[section]
\newtheorem{Rmk}{Remark}[section]

\newcommand{\pd}[1]{\partial_{#1}}
\newcommand{\indFn}[1]{1 \! \! 1_{#1}}
\newcommand{\E}{\mathbb{E}}
\newcommand{\Prb}{\mathbb{P}}

\newcommand{\brak}[1]{\langle #1 \rangle}

\includecomment{commentCoAuthor}

\includecomment{commentToDo}

\newcommand{\arxiv}[1]{#1}
\newcommand{\physD}[1]{}

\begin{document}

\markboth{A. Debussche, N. Glatt-Holtz and R. Temam}
{Local Martingale and Pathwise Solutions for an Abstract Fluids Model}

\markboth{A. Debussche, N. Glatt-Holtz and R. Temam}
{Local Martingale and Pathwise Solutions for an Abstract Fluids Model}

\maketitle

\vskip-4mm

\centerline{\footnotesize{\it $^{\sharp}$IRMAR and ENS Cachan Bretagne }}
\vskip-1mm
\centerline{\footnotesize{\it 35170 Bruz, France}}

\vskip2mm

\centerline{\footnotesize{\it $^\flat$The Institute for Scientific Computing and Applied Mathematics} }
\vskip-1mm
\centerline{\footnotesize{\it Indiana University, Bloomington, IN 47405, USA}}

\begin{center}
\large
\date{\today}
\end{center}

\vskip4mm

\begin{abstract}
  We establish the existence and uniqueness of both local martingale
  and local pathwise solutions of an abstract nonlinear stochastic evolution
  system.  The primary application of this abstract framework is to
  infer the local existence of strong, pathwise solutions to the $3D$
  primitive equations of the oceans and atmosphere forced by a
  nonlinear multiplicative white noise.  Instead of developing our 
  results specifically for the $3D$ primitive equations we choose
  to develop them in a slightly abstract framework which covers 
  many related forms of these equations (atmosphere, oceans, coupled
  atmosphere-ocean, on the sphere, on the $\beta$-plane approximation
  etc and the incompressible Navier-Stokes equations).  In applications,
  all of the details are given for the $\beta$-plane approximation of the oceans
  equations.
\end{abstract}

\section{Introduction}
\label{sec:introduction}

In this work we develop a local existence theory for a class of
abstract stochastic evolution systems of the form:
\begin{equation}\label{eq:AbsEq}
  dU + ( AU + B(U,U) + F(U)) dt = \sigma(U)dW, \quad U(0) = U_0,
\end{equation}
which includes the primitive equations of the ocean as explained 
below.  We choose to develop our setting in a somewhat abstract
framework so that our results may cover some closely related 
equations such as the primitive equations of the atmosphere 
or the coupled atmospheric/oceanic system, equations with chemistry,
equations on the sphere or $\beta$-plane approximations,
etc. We do not expand further on these latter 
applications in this article in order to avoid excessive developments.
The stochastic primitive equations of the ocean have been
previously studied in
\cite{GlattHoltzZiane1, Glattholtz1, GlattHoltzTemam1,
EwaldPetcuTemam, GuoHuang} but none of these works 
address the full 3-d system in the context of
a nonlinear multiplicative noise.

The deterministic primitive equations are widely seen as a fundamental 
model for large scale oceanic and atmospheric systems.  For the oceans
they are derived from the fully compressible Navier-Stokes equations
combined with the Boussinesq and hydrostatic approximations.
See e.g. \cite{Pedlosky} for further physical background.  Given the growing
importance of probabilistic methods in sub-grid scale parameterization,
geophysicists currently devote significant attention to stochastic forms
of the equations of geophysical fluid dynamics; 
see e.g. \cite{EwaldPenland2, 
PenlandSardeshmukh, PenlandEwald1,
Rose1, LeslieQuarini1, MasonThomson,
BernerShuttsLeutbecherPalmer, ZidikheriFrederiksen}.  From the 
mathematical point of view the stochastic primitive equations have
been considered for a two dimensional version of the equations
in \cite{GlattHoltzZiane1,EwaldPetcuTemam,GlattHoltzTemam2}.
In three space dimensions \cite{GuoHuang} has addressed the
case of an additive noise.  In this situation the outcome 
$\omega$ (in the underlying probability space $\Omega$) may be
treated as a parameter in the problem.  Furthermore the methods 
in \cite{GuoHuang} do not allow for physically realistic boundary 
conditions which we are able to treat here with our methods.

Let us recall at this point that the theory of the related stochastic
Navier-Stokes equations have undergone substantial developments; 
see e.g.  \cite{BensoussanTemam, Viot1, Cruzeiro1,
CapinskiGatarek, FlandoliGatarek1, 
ZabczykDaPrato2,
BensoussanFrehse, Breckner, BrzezniakPeszat, FlandoliRomito1, DaPratoDebussche,
MikuleviciusRozovskii2, MikuleviciusRozovskii4, Shirikyan1, FlandoliRomito, GlattHoltzZiane2}.
However to emphasize the notable differences between the stochastic
Navier-Stokes Equations and the primitive equations we recall that
the deterministic Primitive equations are known to be well 
posed in space dimension three 
(see \cite{CaoTiti, Kobelkov, ZianeKukavica}).
This result is not known for the Navier-Stokes 
Equations and is the object of the famous Clay problem.  
On the other hand the primitive equations are 
technically more involved than the Navier- Stokes equations.
For further background concerning the mathematical theory
for the deterministic primitive equations see e.g. the review article
\cite{PetcuTemamZiane}
and the references therein.  The stochastic primitive equations
that we consider in this work are described in detail in 
Section~\ref{sec:examples} below.

In the theory of stochastic evolution equations two notions of solutions are
typically considered namely pathwise (or strong) solutions and martingale (or
weak) solutions.  In the former notion the driving noise is fixed in advance
while in the later case these underlying stochastic elements enter as
an unknown in the problem.    In this work we will consider both notions
for \eqref{eq:AbsEq} and illuminate the relationship between these two
types of solutions. The classical Yamada-Watanabe theorem from
finite dimensional stochastic analysis says that pathwise solutions
exist whenever martingale solutions may be found and pathwise
uniqueness of solutions holds (see e.g. \cite{PrevotRockner}).  Similar
results have later been established along more elementary lines, in
\cite{GyongyKrylov1}, using a simple characterization of convergence
in probability (see Proposition~\ref{thm:GyongyKry} below).  This characterization
may also be employed in the infinite dimensional context and is used 
below to pass from the case of martingale to pathwise solutions.  In
any case to the  best of our knowledge, no one has previously 
established such a `Yamada-Watanabe' type result for the Primitive 
Equations (or for that matter for the $3D$ Navier-Stokes equations). 

The exposition is organized as follows.  In Section~\ref{sec:math-fram}
we make precise the set-up of the abstract problem \eqref{eq:AbsEq}
and briefly recall some relevant mathematical preliminaries from probability
theory and functional analysis.  A Galerkin scheme for
\eqref{eq:AbsEq} is considered in Section~\ref{sec:unif-estim-galerk}.
By making use of an appropriate cut-off function in the formulation
of the equations we are able to establish uniform a-priori estimates 
for the corresponding sequence of approximate solutions.  In
Section~\ref{sec:local-existence} we outline the compactness arguments
that lead to the local existence of martingale solutions.  We turn then
to pathwise solutions in Section~\ref{sec:localmax-pathwisesolns}.
Here the first step is to establish conditions for pathwise uniqueness.
We then revisit the compactness methods described in the previous
section now making use of the additional criteria for convergence in
probability.  In Section~\ref{sec:examples}
we apply the abstract results to the stochastic primitive
equations.  In the
final Sections, \ref{sec:passToLim}, \ref{sec:append-proof-conv}, we
provide the technical details of the passage to the limit from
compactness of Galerkin approximations established earlier in
Sections~\ref{sec:local-existence}, \ref{sec:localmax-pathwisesolns}.

\section{Mathematical Framework}
\label{sec:math-fram}

In this section we set up the abstract system \eqref{eq:AbsEq}, making precise 
the conditions on each of the terms and reviewing the notions of both martingale
and pathwise solutions.  We also recall various results from abstract 
probability theory and functional analysis which play a fundamental role
in the analysis.

\subsection{Abstract Spaces and Operators}
\label{sec:absSpOps}

We begin by fixing a pair of separable Hilbert spaces $H \supset V$,
and assume that the embedding is dense and compact.  We
may thus define the Gelfand inclusions $V \subset H \subset V'$,
where $V'$ is the dual of $V$, relative to $H$.
We denote by $(\cdot,\cdot)$, $|\cdot|$, $((\cdot,\cdot))$ and
$\|\cdot\|$ the norms and inner products of $H$, and  $V$
respectively. The duality product between $V'$ and $V$ is written
$\brak{\cdot,\cdot}$. 

\subsubsection{The Principal Linear Operator}
\label{sec:princ-line-oper}

We now give the precise assumptions on each of the terms appearing
in (\ref{eq:AbsEq}).   They are of course designed to include the case 
of the primitive equations of the ocean, \eqref{eq:PE3DBasic}-
\eqref{eq:basicInitialCond} as we explain below
in Section~\ref{sec:prim-equat-ocean}.
We begin with the linear term supposing that  
$A: D(A) \subset H \rightarrow H$ 
is an unbounded, densely defined, bijective, operator such that
$(AU, U^{\sharp}) = ((U, U^{\sharp}))$  for all  $U, U^{\sharp} \in D(A)$.
As such we see that $A$ is symmetric and may be understood
as bounded operator from $V$ to $V'$ with the duality product given by
\begin{equation}\label{eq:vdualRelA}
  \brak{AU, U^{\sharp}} = ((U, U^{\sharp})), \quad 
  \textrm{ for all } U \in V.
\end{equation}
We further see that $A^{-1}$ is continuous as a map from $H$ into $V$.
Since by assumption $V \subset \subset H$, it follows that $A^{-1}$ is
compact on $H$.  We may also deduce from the given assumptions on $A$
that $A^{-1}$ is symmetric. Applying the classical theory for
symmetric compact operators we infer the existence of a complete orthonormal basis
$\{\Phi_k\}_{k \geq 0}$ for $H$ of eigenfunctions of $A$ so that the associated
sequence of eigenvalues $\{\lambda_k\}_{k \geq 0}$ form an increasing unbounded
sequence.  For the Galerkin scheme below we introduce the finite
dimensional spaces
\begin{displaymath}
  H_n = span \{\Phi_1, \ldots, \Phi_n  \}
\end{displaymath}
and let $P_{n}, Q_{n} = I - P_{n}$ be the projection operators onto $H_n$ and its
orthogonal complement.

Using the basis $\{ \Phi_k\}$ we may also define the fractional powers of
$A$ which are also relevant to the analysis. Given $\alpha > 0,$ take
\begin{displaymath}
  D(A^\alpha) = \left\{
    U \in H:  \sum_k \lambda_k^{2\alpha} |U_k|^2 < \infty \right\}
\end{displaymath}
where $U_k = (U, \Phi_k)$.  On this set we may define $A^\alpha$ according to
\begin{displaymath}
  A^\alpha U  = \sum_{k} \lambda_k^{\alpha} U_k \Phi_{k},
  \quad
  \textrm{ for } U = \sum_{k} U_{k} \Phi_{k}.
\end{displaymath}
Accordingly we equip $D(A^\alpha)$ with the Hilbertian norm
\begin{displaymath}
  |U|_\alpha := |A^\alpha U| = \left( \sum_k \lambda_k^{2\alpha} |U_k|^2 \right)^{1/2}.
\end{displaymath}
Classically we have the following generalized Poincar\'{e} and 
inverse Poincar\'{e} inequalities:
\begin{equation}\label{eq:decompEstimates}
  \begin{split}
     |P_n U |_{\alpha_2} \leq \lambda^{\alpha_2 - \alpha_1}_n |P_n U|_{\alpha_1}, 
     \quad \quad
     |Q_n U |_{\alpha_1} \leq 
      \frac{1}{\lambda^{\alpha_2 - \alpha_1}_n} |Q_n U|_{\alpha_2},\\
  \end{split}
\end{equation}
valid for any $\alpha_1 < \alpha_2$.

Note that as in \cite{Temam4} one may verify that
$D(A^\beta) \subset D(A^\alpha)$ is a compact embedding whenever $\beta
> \alpha$.  Using (\ref{eq:vdualRelA}), 
one may readily verify that $D(A^{1/2}) =
V$ and that $\| U\|= | U |_{1/2}$ for all $U \in V$.  Thus, it is
clear that, in particular, the embedding $D(A) \subset V$ is compact. 

\subsubsection{The Nonlinear Terms}
\label{sec:princ-nonl-term}

We turn next to $B$ which we assume to be a bilinear form mapping $V
\times D(A)$ continuously to $V'$ and $D(A) \times D(A)$
continuously to $H$.  Furthermore we assume the following properties
for $B$:
\begin{equation}\label{eq:Bprop1}
  \begin{split}
  \langle B(U, U^\sharp), U^\sharp \rangle= 0
  \quad \textrm{ for all } U \in V, U^\sharp \in D(A),    
  \end{split}
\end{equation}
\begin{equation}\label{eq:Bprop2}
  |\langle B(U, U^\sharp), U^\flat \rangle|
  \leq c_0 \| U \| | A U^\sharp | \| U^\flat\|
  \quad \textrm{ for all } U, U^\flat \in V, U^\sharp \in D(A),
\end{equation}
\begin{equation}\label{eq:Bprop3}
  \begin{split}
 |\langle B(U, U^\sharp), U^\flat \rangle|
  &\leq c_0 \| U \|^{1/2} | A U |^{1/2} 
  \| U^\sharp \|^{1/2} |A U^\sharp |^{1/2} | U^\flat | \quad
  \textrm{ for all } U, U^\sharp \in D(A), U^\flat \in H.
  \end{split}
\end{equation}
Note that, for brevity of notation, we will sometimes write $B(U)$ for
$B(U,U)$.

We next describe the conditions imposed for $F$ and $\sigma$.  To this 
end we introduce some further notations.
Given any pair of Banach spaces $\mathcal{X}$ and $\mathcal{Y}$ we denote by
$Bnd_u(\mathcal{X}, \mathcal{Y})$, the collection of all continuous mappings 
\begin{displaymath}
  \Psi: [0, \infty) \times \mathcal{X} \rightarrow \mathcal{Y}
\end{displaymath}
so that
\begin{displaymath}
    \| \Psi(x,t) \|_{\mathcal{Y}} \leq c(1 + \|x\|_{\mathcal{X}}), \quad x
    \in \mathcal{X}, t \geq 0\\
\end{displaymath}
where the numerical constant $c$ may be chosen
independently of $t$.  If, in addition,
\begin{displaymath}
   \| \Psi(x,t) - \Psi(y, t)\|_{\mathcal{Y}} \leq c \|x - y \|_{\mathcal{X}},
   \quad x, y \in \mathcal{X}, t \geq 0\\
\end{displaymath}
we say $\Psi$ is in $Lip_u(\mathcal{X}, \mathcal{Y})$.

For $F$ we assume that
\begin{equation}\label{eq:Fmappings}
  \begin{split}
   F: [0, \infty) \times V \rightarrow H.
  \end{split}
\end{equation}
In Section~\ref{sec:local-existence} we assume that
\begin{equation}\label{eq:sublinerConF}
  F \in Bnd_u(V,H).
\end{equation}
Further on in Section~\ref{sec:localmax-pathwisesolns}
\begin{equation}\label{eq:lipConF}
  F \in Lip_u(V,H).
\end{equation}

Similar conditions are also imposed on $\sigma$.  We shall assume throughout this work that
\begin{equation}\label{eq:sigMappings}
  \sigma: [0, \infty) \times H \rightarrow L_2(\mathfrak{U},H).
\end{equation}
Here $\mathfrak{U}$ is an auxiliary Hilbert space and $L_{2}(\mathfrak{U}, H)$ is the
collection of Hilbert-Schmidt operators between $\mathfrak{U}$ and $H$.   See 
Section~\ref{sec:stochasticBackground} for further remarks.
For the case of martingale solutions considered in Section~\ref{sec:local-existence}, we assume that
\begin{equation}\label{eq:BndCondSig}
  \sigma \in Bnd_u(H, L_2(\mathfrak{U}, H)) 
   \cap Bnd_u(V, L_2(\mathfrak{U}, V)) 
   \cap Bnd_u(D(A), L_2(\mathfrak{U}, D(A))).
\end{equation}
On the other hand for pathwise solutions, Section~\ref{sec:localmax-pathwisesolns}, we posit 
\begin{equation}\label{eq:lipCondSig}
  \sigma \in Lip_u(H, L_2(\mathfrak{U}, H)) 
   \cap Lip_u(V, L_2(\mathfrak{U}, V)) 
   \cap Lip_u(D(A), L_2(\mathfrak{U}, D(A))).
\end{equation}

\subsection{The Stochastic Framework}
\label{sec:stochasticBackground}

In order to define the remaining terms in (\ref{eq:AbsEq}), that is 
$\sigma(U) dW$ we must
recall some basic notions and notations from stochastic analysis.  For
more theoretical background on the general theory of stochastic
evolution systems we mention the classical book \cite{ZabczykDaPrato1}
or the more recent treatment in \cite{PrevotRockner}.

To begin we fix a stochastic basis $\mathcal{S}
:= (\Omega, \mathcal{F}, \{\mathcal{F}_t\}_{t \geq 0}, \mathbb{P},
\{W^k\}_{k \geq 1})$,
 that is a filtered probability space with $\{W^k\}_{k
  \geq 1}$ a sequence of independent standard $1$-d Brownian motions
relative to $\mathcal{F}_t$.  In order to avoid unnecessary
complications below we may assume that
$\mathcal{F}_t$ is complete and right continuous (see \cite{ZabczykDaPrato1}).  
Fix a separable
Hilbert space $\mathfrak{U}$ with an associated orthonormal basis $\{e_k\}_{k \geq 1}$. We
may formally define $W$ by taking $W = \sum_k W_k
e_k$.  As such $W$ is a 'cylindrical Brownian' motion evolving over
$\mathfrak{U}$.  

\arxiv{
We next recall some basic definitions and properties of spaces of
Hilbert-Schmidt operators.  For this purpose we suppose that $X$ and
$\tilde{X}$ are any pair of separable Hilbert spaces with the associated 
norms and  inner
products given by $| \cdot |_X$, $| \cdot |_{\tilde{X}}$ and 
$\langle \cdot, \cdot \rangle_{X}$ $\langle \cdot,
\cdot \rangle_{\tilde{X}}$, respectively.  We denote by
\begin{displaymath}
 L_2(\mathfrak{U}, X) = \left\{ R \in \mathcal{L}(\mathfrak{U},X):
   \sum_k |Re_k|^2_{X} < \infty 
   \right\},
\end{displaymath}
the collection of Hilbert-Schmidt operators from $\mathfrak{U}$ to
$X$. By endowing this collection with the inner product
  $\langle R, S \rangle_{L_2(\mathfrak{U}, X)} = \sum_k \langle R e_k, S e_k \rangle_X$,
we may consider $L_2(\mathfrak{U},X)$ as itself being a Hilbert space.
Note that when $R \in L_{2} (\mathfrak{U}, X)$ we shall often denote
$R_k = R e_k$ and we may therefore associate $R$ with the sequence 
$\{R_k\}_{k \geq 1}$.
One may readily show that if $R^{(1)} \in L_2(\mathfrak{U}, X)$ and
$R^{(2)} \in L(X,\tilde{X})$ then indeed $R^{(2)}R^{(1)} \in
L_2(\mathfrak{U}, \tilde{X})$.  

We also define the auxilary space $\mathfrak{U}_0 \supset \mathfrak{U}$ via
\begin{displaymath}
  \mathfrak{U}_0 
  := \left\{ v = \sum_{k \geq 0} \alpha_{k} e_{k} : 
  \sum_k {\alpha^{2}_{k}}{k^2} < \infty \right\},
\end{displaymath}
endowed with the norm
  $| v |_{\mathfrak{U}_{0}}^{2} := \sum_{k} \frac{\alpha^{2}_{k}}{k^{2}}, \quad v = \sum_k \alpha_{k} e_{k}$.
Note that the embedding of $\mathfrak{U} \subset \mathfrak{U}_0$ 
is Hilbert-Schmidt. Moreover, using standard Martingale arguments 
with the fact that each $W_{k}$ is almost surely continuous 
(see \cite{ZabczykDaPrato1}) we have that, for almost every $\omega \in \Omega$,
  $W(\omega) \in C([0,T], \mathfrak{U}_0)$.
}

\physD{Fix a second separable Hilbert space $X$ with its associated
norm denoted by $| \cdot |_{X}$.  Let $L_{2}(\mathfrak{U}, X)$ be the collection
of Hilbert-Schmidt operators between $\mathfrak{U}$ and $X$, endowed with
the usual norm.}
Given an $X$-valued predictable\footnote{For a given  stochastic basis
$\mathcal{S}$, let $\Phi = \Omega \times [0,\infty)$ and take
$\mathcal{G}$ to be the $\sigma$-algebra generated by sets of the
form
\begin{displaymath}
    (s,t] \times F, \quad 0 \leq s< t< \infty, F \in \mathcal{F}_s;
    \quad \quad
    \{0\} \times F, \quad F \in \mathcal{F}_0.
\end{displaymath}
Recall that an $X$ valued process $U$ is called predictable (with
respect to the stochastic basis $\mathcal{S}$) if it is measurable
from $(\Phi,\mathcal{G})$ into $(X, \mathcal{B}(X))$,
$\mathcal{B}(X)$ being the family of Borel sets of $X$.}
process
$G \in L^{2}(\Omega; L^{2}_{loc}
([0, \infty),L_{2}(\mathfrak{U}, X)))$ 
one may define the 
(It\={o}) stochastic integral
\begin{displaymath}
   M_{t} := \int_{0}^{t} G dW = \sum_k \int_0^t G_k dW_k,
   \physD{ \quad \textrm{ where } G_{k} = G e_{k}, }
\end{displaymath}
as an element in $\mathcal{M}^2_X$, that is the space of all
$X$-valued square integrable martingales (see \cite[Section
2.2, 2.3]{PrevotRockner}).  As such $\{M_t \}_{t \geq 0}$ has
many desirable properties.  Most notably for the analysis here,
the Burkholder-Davis-Gundy inequality holds which in the present
context takes the form,
\begin{equation}\label{eq:BDG}
  \E \left(\sup_{t \in [0,T]} \left| \int_0^t G dW  \right|_X^r \right)
  \leq c \E \left(
    \int_0^T |G|_{L_2(\mathfrak{U}, X)}^2 dt \right)^{r/2},
\end{equation}
valid for any $r \geq  1$.  Here $c$ is an absolute constant depending
only on $r$.
We shall also make use of a variation of this inequality, established
in \cite{FlandoliGatarek1} which applies to fractional derivatives of
$M_t$.   For $p \geq 2$ and $\alpha \in [0,1/2)$ we have
\begin{equation}\label{eq:BDGfrac}
  \E \left(  
     \left| \int_0^{t} G dW \right|_{W^{\alpha, p}([0,T];X)}^p \right)
     \leq c \E \left(
    \int_0^T |G|_{L_2(\mathfrak{U}, X)}^p dt \right),
\end{equation}
which holds for all $X$-valued predictable $G \in L^{p}(\Omega; L^{p}_{loc}([0,
\infty),L_{2}(\mathfrak{U}, X)))$.  For the convenience of the reader,
we shall recall the definition of
the spaces $W^{\alpha, p}([0,T], X)$ in
Section~\ref{sec:compEmbedding} below.

\begin{Rmk}\label{rmk:PossibleNoiseStructures}
Under the assumptions, \eqref{eq:BndCondSig}, \eqref{eq:lipCondSig}, 
on $\sigma$, the
stochastic integral $t \mapsto \int_0^t \sigma(U) dW$ may be shown to
be well defined (in the It\={o} sense), taking values in $H$ whenever
$U \in L^{2}(\Omega, L^{2}_{loc}([0,\infty); H))$ and is predictable.
Such terms may be seen
to cover a wide class of examples, including but not
limited to the classical cases of additive and linear 
multiplicative noise, projections of the solution in any
direction, and directional forcings of Lipschitz functionals 
of the solution.  See e.g. \cite{GlattHoltzZiane2} for further details.
\end{Rmk}

In Section~\ref{sec:append-proof-conv} we establish
the following convergence theorem for stochastic integrals.  This result
will be used below to facilitate the passage to the limit in the
Galerkin scheme.  The statement and
proof generalizes ideas found in \cite{Bensoussan1}.
\begin{Lem}\label{thm:ConvThm}
  Let $(\Omega, \mathcal{F}, \mathbb{P})$ be a fixed probability space,
  $X$ a separable Hilbert space. Consider a sequence of stochastic
  bases $\mathcal{S}_n = (\Omega, \mathcal{F},\{\mathcal{F}_t^n \}_{t
    \geq 0}, \mathbb{P}, W^n)$, that is a sequence so that each $W^n$ is
  cylindrical Brownian motion (over $\mathfrak{U}$) with respect to
  $\mathcal{F}_t^n$. Assume that $\{G^{n}\}_{n \geq 1}$ are a
  collection of $X$-valued $\mathcal{F}_t^n$ predictable processes
  such that $G^{n} \in L^2([0,T], L_2(\mathfrak{U}, X))$ a.s.
 Finally consider $\mathcal{S} = (\Omega, \mathcal{F},\{\mathcal{F}_t\}_{t
    \geq 0}, \mathbb{P}, W)$ and $G  \in L^2([0,T],
  L_2(\mathfrak{U}, X))$, which is $\mathcal{F}_t$ predictable. If  
  \begin{subequations}\label{eq:ConvAsmp}
    \begin{gather}
      W^{n} \rightarrow W  \quad
      \textrm{ in probability in }
      C([0,T], \mathfrak{U}_0),
      \label{eq:ConvAsmpNoise}\\
      G^{n} \rightarrow G \quad
      \textrm{ in probability in }
      L^2([0,T]; L_2(\mathfrak{U}, X)),
      \label{eq:ConvAsmpStochIntegrand}
    \end{gather}
  \end{subequations}
  then
  \begin{equation}\label{eq:StochasticVitaliConc}
    \int_0^t G^n dW^n \rightarrow \int_0^t G dW
    \quad \textrm{ in probability in } L^2([0,T]; X).
  \end{equation}
  \end{Lem}

\physD{
In order to pass from martingale to pathwise solutions we make essential
use of an elementary but powerful characterization of convergence in 
probability as given in \cite{GyongyKrylov1}.
Suppose that $\{Y_n \}_{n \geq 0}$ is a sequence of $X$-valued random
variables on a probability space $(\Omega, \mathcal{F}, \mathbb{P})$.
Let $\{\mu_{n,m}\}_{n, m \geq 1}$ be the collection of joint laws of
$\{Y_n \}_{n \geq 1}$, that is
\begin{displaymath}
    \mu_{n,m}(E) := \Prb( (Y_n, Y_m) \in E),
    \quad E \in \mathcal{B}( X \times X ).
\end{displaymath}
The result from \cite{GyongyKrylov1} is the following:
\begin{Prop}\label{thm:GyongyKry}
  A sequence of $X$ valued random variables $\{Y_n\}_{n \geq 0}$ converges 
  in probability if and only if for
  every subsequence of joint probabilities laws, $\{\mu_{n_k, m_k} \}_{ k \geq 0}$, 
  there exists a further subsequence which converges weakly to a probability 
  measure $\mu$ such that
  \begin{equation}\label{eq:diagonalCondGK}
    \mu( \{(x,y) \in X \times X: x = y  \}) = 1.
  \end{equation}
\end{Prop}
}

Finally we describe the assumptions for the initial condition
$U_0$ which may be random in general.  In 
Section~\ref{sec:local-existence}, where we consider
the case of Martingale solutions, since the stochastic basis is an
unknown of the problem we are only able to
specify $U_0$ as an initial probability measure $\mu_0$ on $V$ such that:
\begin{equation}\label{eq:U0CondMartingale}
  \int_V \|U\|^q d \mu_0(U) < \infty.
\end{equation}
Here $q \geq 2$ will be specified below, see Theorem~\ref{thm:MainResult}
as well as Lemma~\ref{thm:UniformEstimateGalerkin}.
On the other hand for pathwise solutions where the stochastic basis
$\mathcal{S}$ is fixed we assume that relative to this basis $U_0$ is
a $V$ valued random variable such that
\begin{equation}\label{eq:U0Cond1Pathwise}
  U_0 \in L^2(\Omega;V) 
  \textrm{ and is } \mathcal{F}_0 \textrm{ measurable.}
\end{equation}

\subsection{Definition of Solutions}
\label{sec:defOfSolns}

We next give the definitions of local and global
solutions of \eqref{eq:AbsEq} for both Martingale and Pathwise
Solutions.
\begin{Def}[Local and Global Martingale Solutions]
  \label{def:MartingaleSoln}
  Suppose $\mu_0$ is probability measure on $V$ satisfying
  \eqref{eq:U0CondMartingale} with $q \geq 8$ 
  and assume that (\ref{eq:sublinerConF})
  and (\ref{eq:BndCondSig})
  hold for $F$ and $\sigma$ respectively.
  \begin{itemize}
  \item[(i)] A triple $(\mathcal{S}, U, \tau)$ is a \emph{local
      Martingale solution} if $\mathcal{S} = (\Omega, \mathcal{F}, \{\mathcal{F}_t\}_{t \geq
    0}, \mathbb{P}, W)$ is a stochastic basis, $\tau$ is stopping time relative to
  $\mathcal{F}_t$ and $U(\cdot) = U(\cdot \wedge \tau): \Omega \times
  [0,\infty) \rightarrow V$ is an $\mathcal{F}_{t}$ adapted process such that:
  \begin{equation}\label{eq:regularityCondMG}
      \begin{split}
        U(\cdot \wedge \tau) \in  L^2(\Omega; C ([0,\infty); V)),\\
        U \indFn{t \leq \tau} \in  L^2(\Omega;
        L^2_{loc}([0,\infty); D(A)));
      \end{split}
  \end{equation}
  the law of $U(0)$ is $\mu_0$ i.e. $\mu_0(E) = \Prb( U(0) \in E)$, for all Borel subsets $E$ of $V$, and
 $U$ satisfies for every $t \geq 0$,
  \begin{equation}\label{eq:spdeAbstracMG}
      \begin{split}
        U(t \wedge \tau)
           + \int_0^{t \wedge \tau} (A U +  B(U) + F(U)) ds
           = U(0)
            + \int_0^{t \wedge \tau} \sigma(U) dW,
     \end{split}
  \end{equation}
  with the equality understood in $H$.
  \item[(ii)] We say that the (Martingale) solution $(\mathcal{S}, U,
  \tau)$ is global if $\tau = \infty$, $\Omega$-a.s.
  \end{itemize}
\end{Def}

We next define pathwise solutions of \eqref{eq:AbsEq} where the
stochastic basis is fixed in advance.
\begin{Def}[Local, Maximal and Global Pathwise Solutions]
  \label{def:Pathwisesolns}
  Let $\mathcal{S} = (\Omega, \mathcal{F}, \{\mathcal{F}_t\}_{t \geq
    0}, \mathbb{P}, W)$ be a fixed stochastic basis and suppose that
  $U_0$ is an $V$ valued random variable (relative to $\mathcal{S}$)
  satisfying (\ref{eq:U0Cond1Pathwise}).  Assume that $F$ satisfies
  (\ref{eq:lipConF}) and that
  (\ref{eq:lipCondSig}) holds for $\sigma$.
  \begin{itemize}
  \item[(i)] A pair $(U, \tau)$ is a \emph{a local pathwise solution} of
    (\ref{eq:AbsEq}) if $\tau$ is a strictly positive stopping
    time and $U(\cdot \wedge \tau)$ is an $\mathcal{F}_{t}$-adapted
    process in $V$ so that (relative to the fixed basis $\mathcal{S}$)
    \eqref{eq:regularityCondMG}, \eqref{eq:spdeAbstracMG}
    hold.
  \item[(ii)] Pathwise solutions of \eqref{eq:AbsEq} are said to be
   \emph{(pathwise) unique} up to a stopping time $\tau > 0$ if given
   any pair of pathwise solutions $(U^1,\tau)$ and $(U^2,\tau)$ which
   coincide at $t= 0$ on a subset $\tilde{\Omega}$ of $\Omega$,
   $\tilde{\Omega} = \{U^1(0) = U^2(0)\}$, then
   \begin{displaymath}
    \Prb \left( \indFn{\tilde{\Omega}} ( U^1(t \wedge \tau) - U^2(t \wedge
      \tau))  = 0; \forall t \geq 0 \right) = 1.
   \end{displaymath}
\item[(iii)] Suppose that $\{\tau_n\}_{n\geq 1}$ is a strictly increasing sequence
    of stopping times converging to a (possibly infinite) stopping time
    $\xi$ and assume that $U$ is a predictable process in $H$.  We
   say that the triple $(U,\xi, \{\tau_n\}_{n\geq 1} )$ is \emph{a maximal
     strong solution} if $(U, \tau_n)$ is a local strong solution for
   each $n$ and
  \begin{equation}\label{eq:FiniteTimeBlowUp}
      \sup_{t \in [0, \xi]} \|U\|^2 + \int_0^{\xi} |A U|^2 ds = \infty
  \end{equation}
  almost surely on the set $\{\xi < \infty\}$.  If, moreover
  \begin{equation}\label{eq:blowUpannouncement}
      \sup_{t \in [0, \tau_n]} \|U\|^2 + \int_0^{\tau_n} |A U|^2 ds = n,
  \end{equation}
  for almost every $\omega \in \{ \xi < \infty \}$ then the sequence $\tau_n$ 
  is said to \emph{announce} any finite time blow up.
  \item[(iii)] If $(U, \xi)$ is a maximal strong solution and $\xi =
    \infty$ a.s. then we say that the solution is global.
  \end{itemize}
\end{Def}

We may now state precisely the main results in the work:
\begin{Thm}\label{thm:MainResult}
  \mbox{}
  \begin{itemize}
  \item[(i)]  Suppose that $\mu_0$ satisfies \eqref{eq:U0CondMartingale},
    for $q \geq 8$ and that $F$ and $\sigma$ maintain
    \eqref{eq:sublinerConF},  \eqref{eq:BndCondSig} respectively.
    Then there exists a local Martingale solution $(\mathcal{S}, U, \tau)$ of (\ref{eq:AbsEq}).
  \item[(ii)] Assume that, relatively to a fixed stochastic basis $\mathcal{S}$, 
    $U_0$ satisfies \eqref{eq:U0Cond1Pathwise} and that 
    $F$ and $\sigma$ fulfill \eqref{eq:lipConF} and \eqref{eq:lipCondSig}.
    Then there exists a unique, maximal pathwise solution, $(U,\xi, \{\tau_n\}_{n\geq 1} )$,
    of (\ref{eq:AbsEq}).
 \end{itemize}
\end{Thm}
The compactness arguments leading to Theorem~\ref{thm:MainResult} 
are carried out in Sections~\ref{sec:local-existence}
and \ref{sec:localmax-pathwisesolns} for (i) and (ii) respectively.
We provide the details of the passage to the limit needed for 
both items in Section~\ref{sec:passToLim}.

\begin{Rmk}\label{Rmk:ProbDependence}
  \begin{itemize}
  \item[(i)] We note that, as we are working at the intersection of
    two fields, the terminology may cause some confusion.  In the
    literature for stochastic differential equations the term ``weak
    solution'' is sometimes used synonymously with the term
    ``martingale solution'' while the designation ``strong solution''
    may be used for a ``pathwise solution''.  See the introductory
    text of {\O}ksendal \cite{Oksendal1} for example. The former
    terminologies are avoided here because it is confusing in the
    context of partial differential equations.  Indeed, from the PDE
    point of view, strong solutions are solutions which are uniformly
    bounded in $H^1$, while weak solutions are those which are merely
    bounded in $L^2$.  In this work we are therefore considering both
    weak and strong solutions in probabilistic sense.  From the PDE
    point of view we may say that we are considering strong type
    solutions since, in the applications considered here $V$ is taken
    to be an appropriate subspace of $H^{1}$ that incorporates the
    boundary conditions, etc., for \eqref{eq:PE3DBasic}.
  \item[(ii)] The notion of global existence, both for the Martingale
    and the Pathwise contexts, are included here for the sake of
    completeness.  Of course, the passage from the maximal to the
    global existence of pathwise solutions is a significant further
    step in the analysis and requires further structure for
    \eqref{eq:AbsEq}.  We refer the reader to \cite{GlattHoltzZiane2}
    and \cite{GlattHoltzTemam1} where this is done for the 2D
    Navier-Stokes Equations and the 2D Primitive Equations
    respectively. Current work, making use of the main result herein,
    treats the global existence of solutions for the stochastic 3D
    Primitive Equations \cite{DebusscheGlattHoltzTemamZiane1}.
  \item[(iii)] In Section~\ref{sec:localmax-pathwisesolns},
    \ref{sec:passToLim} we consider both Martingale and 
    Pathwise solutions of the modified system
    \begin{equation}\label{eq:modFullSystemCut}
       dU + ( AU + \theta(\|U - U_*\|)B(U) + F(U)) dt = \sigma(U)dW, 
        \quad U(0) = U_0,\\
    \end{equation}
    where
    \begin{equation}\label{eq:linearAuxSystem}
      \frac{d}{dt} U_* + AU_* = 0, \quad U(0) = U_0,\\
    \end{equation}
    and $\theta$ is a smooth cut-off function as defined below in
    \eqref{eq:CuttoffDef}.  The notions of solutions for
    \eqref{eq:modFullSystemCut} are, with trivial modifications,
    identical to Definitions~\ref{def:MartingaleSoln},
    \ref{def:Pathwisesolns} given for \eqref{eq:AbsEq} above.
\end{itemize}
\end{Rmk}

\subsection{Compact Embedding Theorems}
\label{sec:compEmbedding}
We shall make use of two compact embedding results taken from
\cite{FlandoliGatarek1} which we restate here.  
See also related results in \cite{Temam4}.
To this end we first recall some spaces of fractional (in time) derivative.  Such 
spaces are natural since we do not expect solutions of stochastic evolution
systems to be differentiable in time but merely Holder continuous of
order strictly less than $1/2$.  

Let $X$ be a separtable Hilbert space and
denote the associated norm by $|\cdot |_X$.  For fixed $p> 1$ and
$\alpha \in (0,1)$ we define
\begin{displaymath}
 W^{\alpha, p} ([0,T]; X) :=
 \left\{ U \in L^p([0,T]; X); 
    \int_0^T \int_0^T 
    \frac{|U(t') - U(t'')|_{X}^p}{|t'-t''|^{1 + \alpha p}}dt' dt''
        < \infty   \right\}.
\end{displaymath}
We endow this space with the norm
\begin{displaymath}
  | U |_{W^{\alpha,p}([0,T]; X)}^p :=  
  \int_0 ^T|U(t')|^p_X dt' +
  \int_0^T \int_0^T  
  \frac{|U(t') - U(t'')|_{X}^p}{|t'-t''|^{1 + \alpha p}}dt' dt''.
\end{displaymath}
For the case when $\alpha = 1$ we take
$W^{1, p} ([0,T]; X) :=
 \{ U \in L^p([0,T]; X); 
           \frac{dU}{dt} \in L^p([0,T]; X) \}$,
to be the classical Sobolev space with its usual norm
\begin{displaymath}
  | U |_{W^{1,p}([0,T]; X)}^p :=  
  \int_0^T |U(t')|^p_X + \left| \frac{dU}{dt}(t') \right|_X^{p} dt'.
\end{displaymath}
Note that for $\alpha \in (0,1)$, $W^{1,p}([0,T]; X) \subset W^{\alpha, p} ([0,t]; X)$ and
   $| U |_{W^{\alpha,p}([0,T]; X)} \leq C | U |_{W^{1,p}([0,T]; X)}$.
With these preliminaries in hand we may now state the compact
embeddings needed below (see \cite{FlandoliGatarek1})
\begin{Lem}\label{thm:FlanGatComp}
  \mbox{}
  \begin{itemize}
  \item[(i)] Suppose that $X_{2} \supset X_{0} \supset X_{1}$ are
    Banach spaces with $X_{2}$ and $X_{1}$ reflexive, and the
    embedding of $X_{1}$ into $X_{0}$ compact.  Then for any $1 < p <
    \infty$ and $0 < \alpha < 1$, the embedding: 
    \begin{equation}\label{eq:cmptEmbdingLpspace}
     L^p([0,T]; X_{1}) \cap W^{\alpha,p}([0,T]; X_{2})       
       \subset \subset L^p([0,T];X_{0})
    \end{equation}
    is compact.
  \item[(ii)] Suppose that $Y_0 \supset Y$ are Banach spaces with $Y$
    compactly embedded in $Y_0$.  Let 
    $\alpha \in (0,1]$ and $p \in (1,\infty)$ be such that
    $\alpha p > 1$ then
    \begin{equation}\label{eq:fracSobEmbeddingContFns}
      W^{\alpha, p}([0,T]; Y) \subset \subset C([0,T],Y_0)
   \end{equation}
    and the embedding is compact.
  \end{itemize}
\end{Lem}

\arxiv{
\subsection{Some Tools From Abstract Probability Theory}
\label{sec:ConvThmsAbs}

We next review some classical convergence results for probability
measures defined on separable metric spaces.
In conjuction with the embeddings given in Section~\ref{sec:compEmbedding}, these results
provide some powerful means to address the difficulty of establishing
compactness for the collection of Galerkin approximations associated to 
\eqref{eq:AbsEq}.

Let $(X,d)$ be a complete separable metric space and take $\mathcal{B}(X)$ to be
the associated borel $\sigma$-algebra.   Also, we define $C_b(X)$ to
be the collection of all real valued continuous bounded functions on $X$ and take
$Pr(X)$ to be the set of all probability measures on $(X,
\mathcal{B}(X))$. 
Recall that a collection $\Lambda \subset Pr(X)$ is
said to be \emph{tight} if, for every $\epsilon > 0$, there exists a compact
set $K_\epsilon \subset X$ such that:
\begin{displaymath}
  \mu ( K_\epsilon) \geq 1 - \epsilon \quad \textrm{ for all } \mu \in \Lambda.
\end{displaymath}
On the other hand a sequence $\{ \mu_n \}_{n \geq 0} \subset Pr(X)$ is said to
\emph{converge weakly} to an probability measure $\mu$ if
\begin{displaymath}
  \int f d \mu_n \rightarrow \int f d \mu
\end{displaymath}
over all $f \in C_b(X)$.  We say that a set $\Lambda \subset Pr(X)$ is
weakly compact if every sequence $\{ \mu_n \} \subset \Lambda$ possesses a
weakly convergent subsequence. 

Proofs of the following classical results may be found in e.g.
\cite{ZabczykDaPrato1}.
\begin{Prop}\label{thm:ProhorovSkorhod}
  \begin{itemize}
  \item[(i)] A collection $\Lambda \subset Pr(X)$ is weakly compact if and
    only if it is tight.
  \item[(ii)]  Suppose that a sequence $\{ \mu_n \}_{n \geq 1}$
    converges weakly to a measure $\mu$.  Then there exists a
    probability space $(\tilde{\Omega}, \tilde{\mathcal{F}}, \tilde{\Prb})$ 
    and a sequence of $X$ valued random variables $\{ \tilde{Y}_n \}_{n \geq 0}$ 
    (relative to this space)
    such that $\tilde{Y}_n$ converges almost 
    surely to the random variable $\tilde{Y}$ and
    such that the laws of $\tilde{Y}_n$ and $\tilde{Y}$ are $\mu_{n}$ and $\mu$, i.e.
    $\mu_{n} (E) = \Prb(Y_n \in E)$,  $\mu(E) = \Prb(Y \in E)$,
    for all  $E \in \mathcal{B}(X)$.
  \end{itemize}
\end{Prop}
Finally we come to an elementary but powerful characterization of convergence in 
probability introduced in \cite{GyongyKrylov1}.
Suppose that $\{Y_n \}_{n \geq 0}$ is a sequence of $X$-valued random
variables on a probability space $(\Omega, \mathcal{F}, \mathbb{P})$.
Let $\{\mu_{n,m}\}_{n, m \geq 1}$ be the collection of joint laws of
$\{Y_n \}_{n \geq 1}$, that is
\begin{displaymath}
    \mu_{n,m}(E) := \Prb( (Y_n, Y_m) \in E),
    \quad E \in \mathcal{B}( X \times X ).
\end{displaymath}
The result from \cite{GyongyKrylov1} is the following:
\begin{Prop}\label{thm:GyongyKry}
  A sequence of $X$ valued random variables $\{Y_n\}_{n \geq 0}$ converges 
  in probability if and only if for
  every subsequence of joint probabilities laws, $\{\mu_{n_k, m_k} \}_{ k \geq 0}$, 
  there exists a further subsequence which converges weakly to a probability 
  measure $\mu$ such that
  \begin{equation}\label{eq:diagonalCondGK}
    \mu( \{(x,y) \in X \times X: x = y  \}) = 1.
  \end{equation}
\end{Prop}
}

\section{The Approximation Scheme}
\label{sec:unif-estim-galerk}

We now implement a Galerkin scheme for (\ref{eq:AbsEq}).  
To this end we introduce the projected operators
\begin{displaymath}
  B^n(U) = P_n B(U), \quad
  F^n(U) = P_n F(U), \quad
  \sigma^n(U) = P_n\sigma(U),
\end{displaymath}
where $U \in V$.  We shall also make use of a 'cut-off' function $\theta:
\mathbb{R} \rightarrow [0,1]$, 
which is $C^\infty$ and such that:
\begin{equation}\label{eq:CuttoffDef}
  \theta(x) = 
  \begin{cases}
    1  & \textrm{ if } | x | \leq \kappa,\\
    0  & \textrm{ if } | x | \geq 2\kappa.\\
  \end{cases}
\end{equation}
Here we choose $\kappa$ to be any positive constant, independent of $n$, such that
\begin{equation}\label{eq:kappaCond}
	\kappa \leq \frac{1}{64 c_0},
\end{equation}
where $c_0$ is the constant appearing in \eqref{eq:Bprop3}.  The 
reason for this choice will be made apparent in the proof of Lemma~\ref{thm:UniformEstimateGalerkin}
(see \eqref{eq:ReallyFuckedCutoff2}, \eqref{eq:nonlinearEstMods}).

We now fix a stochastic basis $\mathcal{S} = (\Omega, \mathcal{F}, \{\mathcal{F}_t\}_{t \geq
  0}, \mathbb{P}, W)$ and an element $U_{0} \in V$ 
with law $\mu_{0}$. We find pathwise
solutions to the Galerkin systems defined by \eqref{eq:GalerkinCutoff} relative 
to this basis below.
Since we allow for an ill-behaved nonlinear term $B$ (that is satisfying
\eqref{eq:Bprop1}, \eqref{eq:Bprop2}, \eqref{eq:Bprop3})
we introduce an auxiliary linear system in order to carry uniform
estimates on the Galerkin systems.  We take $U^n_*$ to be the unique
($H_{n}$ valued) solution of
\begin{equation}  \label{eq:AuxLinear}
  \frac{d}{dt} U^n_* + A U^n_* = 0, \quad U^n_*(0) = P_n U_0.
\end{equation}
One may readily verify that, for any $p \geq 2$, $U^n_*$ satisfies the estimates
\begin{equation}\label{eq:LpTimeEstimatesUstar}
  \begin{split}
  \sup_{t' \leq t} \| U^n_*\|^p &+ \int_0^T |AU^n_*|^2 \| U^n_*\|^{p-2}
  dt'  +\left(\int_0^T |AU^n_*|^2 dt' \right)^{p/2} \leq c\| U_0\|^p.\\
\end{split}
\end{equation}
It is also clear that
\begin{equation}\label{eq:timeDerivative}
  |U^n_*|_{W^{1,2}(0,T;H)} \leq c \int_0^T |AU^n_*|^2 dt'  \leq c \|U_0\|^2.
\end{equation}
With these notations in place we define the \emph{Galerkin system at
  order $n$}
\begin{equation}\label{eq:GalerkinCutoff}
  \begin{split}
    d U^n + [A U^n + \theta( \| U^n - U^n_*  \| ) &B^n(U^n) + F^n(U^n) ]dt
     =\sigma^n(U^n) dW,\\
     U^n(0) &= P_n U_0 := U^n_0.
 \end{split}
\end{equation}
Here $U^n$ is an adapted process in
$C([0,T]; H_n)
\cong C([0,T], \mathbb{R}^n)$. The $U^n_*$ appearing in the cutoff
function $\theta$ are solutions of the linear systems
(\ref{eq:AuxLinear}).  The significance of this addition will become
clear in the proof of Lemma~\ref{thm:UniformEstimateGalerkin} below.
Note that, due to the preserved cancellation property in the nonlinear
portion of the equation, the existence and uniqueness of solutions at
each order is standard.  See, for example, \cite{Flandoli1} for further details.
\begin{Lem}\label{thm:UniformEstimateGalerkin}
  Assume that $F$ and $\sigma$ satifisfy \eqref{eq:sublinerConF}, \eqref{eq:BndCondSig}.
  Fix $U_{0}$, a $V$-valued, $\mathcal{F}_{0}$ measurable random
  variable and consider the associated sequence of solutions
  $\{U^{n}\}_{n \geq 1}$ of the Galerkin system \eqref{eq:GalerkinCutoff},
  \eqref{eq:AuxLinear}.  We suppose that the constant $\kappa$ appearing
  in the cutoff function $\theta$ satisfies \eqref{eq:kappaCond}.
 Let $p \geq 2$ and suppose that
    \begin{equation} \label{eq:DataMomentConditionMartingale}
      \E \|U_0\|^q < \infty  
      \textrm{ for some } q \geq \max\{2p, 4\}.
    \end{equation}
    Then there exists a finite number $K >0$ depending only on $p$, $\E \|U_{0}\|^{q}$ and the 
    the constants in \eqref{eq:Bprop3}, \eqref{eq:sublinerConF}, \eqref{eq:BndCondSig} 
    such that 
  \begin{itemize}
  \item[(i)] for every $n \geq 1$,
 \begin{equation}    \label{eq:UniformGalEst1}
     \E  \left( 
       \sup_{t' \leq T} \|U^n \|^p + 
      \int_0^T |A U^n|^2 \|U^n \|^{p-2} dt'
      \right) \leq K,
    \end{equation}
    and also
 \begin{equation}    \label{eq:UniformGalEst2}
     \E  \left( 
     \int_0^T |A U^n|^2 dt'
      \right)^{p/2} \leq K,
  \end{equation}  
and finally
 \begin{equation}\label{eq:uniformFracTmEstGalerkin}
    \E \left(  
     \left| \int_0^{t} \sigma^n(U^n) dW \right|_{W^{\alpha, p}([0,T];H)}^p
        \right)  \leq K.
 \end{equation}
\item[(ii)] If under the given assumptions we additionally suppose that $p \geq 4$, then we have, for all
$n\geq 1$:
    \begin{equation}\label{eq:UniformDrftyEst}
     \E \left(  
     \left| U^n(t) - \int_0^t\sigma^n(U^n) dW\right|_{W^{1, 2}([0,T];H)}^2
        \right)  \leq K.
  \end{equation} 
 \end{itemize}
\end{Lem}
\begin{proof}
  Define $\bar{U}^n := U^n - U^n_*$.   We may readily observe that $\bar{U}^n$ satisfies 
  \begin{equation}\label{eq:smallDataSystem}
    \begin{split}
      d \bar{U}^n+ [A  \bar{U}^n 
           + \theta( \| \bar{U}^n \|) B^n( \bar{U}^n + U^n_*) 
           &+ F^n( \bar{U}^n + U^n_*)] dt
           =  \sigma^n (\bar{U}^n + U^n_*) dW,\\
      \bar{U}^n(0) &= 0.\\
    \end{split}
  \end{equation}
  We apply $A^{1/2}$ to this system.  With the It\={o} formula
  we infer, for $p \geq 2$ that,
  \begin{equation}\label{eq:ItoEvolutoinShiftedEqn}
    \begin{split}
      d \| \bar{U}^n&\|^p+ p|A \bar{U}^n|^2 \| \bar{U}^n\|^{p-2}dt\\
         =& - p\langle F^n( \bar{U}^n + U^n_*), A \bar{U}^n \rangle 
              \| \bar{U}^n\|^{p-2}dt
             + \frac{p}{2} |\sigma^n (\bar{U}^n +U^n_*)|^2_{L_2(\mathfrak{U}, V)} 
                 \| \bar{U}^n\|^{p-2}dt\\
          &+ \frac{p(p-2)}{2} \langle \sigma^n (\bar{U}^n + U^n_*)  , A \bar{U}^n \rangle^2
          \| \bar{U}^n\|^{p-4}dt
         - p \theta( \| \bar{U}^n \|) \langle B^n( \bar{U}^n + U^n_*)
                  , A \bar{U}^n \rangle  \| \bar{U}^n\|^{p-2}dt\\
         &+ p \| \bar{U}^n\|^{p-2} 
                \langle \sigma^n (\bar{U}^n + U^n_*) 
                   , A \bar{U}^n \rangle  dW\\
         =:& (J_{1}^p + J_{2}^p + J_{3}^p + J_{4}^p) dt + J_{5}^p dW.
   \end{split}
  \end{equation}
  We are able to estimate the first four deterministic terms pointwise in time.  Using
  \eqref{eq:sublinerConF} we observe that
  \begin{equation}    \label{eq:sublinearEst1}
   \begin{split}
     |J_1^p| \leq&c (1 + \|U^n_* \| + \| \bar{U}^n\|) |A \bar{U}^n|  
                      \|\bar{U}^n\|^{p-2}
             \leq \frac{p}{8} |A \bar{U}^n|^2 \|\bar{U}^n\|^{p-2}
                      + c (1 + \| U^n_*\| + \| \bar{U}^n\|)^2 \|\bar{U}^n\|^{p-2} \\
            \leq& \frac{p}{8} |A \bar{U}^n|^2 \|\bar{U}^n\|^{p-2}
                      + c (1 + \| U^n_* \|)^p
                     + c\|\bar{U}^n\|^p.
   \end{split}
  \end{equation}
  The terms $J_2^p$ and $J_3^p$ are also estimated directly using \eqref{eq:BndCondSig}
  \begin{equation}    \label{eq:easyEstJ3J4}
   \begin{split}
      |J_2^p| + |J_3^p| \leq&  c((1 + \|U^n_* \|)^2 + \| \bar{U}^n\|^2) \|\bar{U}^n\|^{p-2}
                      \leq c (1 + \| U^n_* \|)^p + c\|\bar{U}^n \|^p.
  \end{split}
  \end{equation}
  Using the bilinearity of $B$ the term $J_4^p$ splits according to:
  \begin{equation}    \label{eq:nonlinearBspliting}
   \begin{split}
    | J_4^p | \leq& p
    \theta( \| \bar{U}^n \|)  \| \bar{U}^n\|^{p-2} |A \bar{U}^n|
    ( | B( U^n_*)| + |B( \bar{U}^n,U^n_*)| +
      |B( U^n_*, \bar{U}^n)| + |B(\bar{U}^n)|)\\
      :=& J_{4,1}^p +J_{4,2}^p +J_{4,3}^p +J_{4,4}^p.
    \end{split}
  \end{equation}
  We estimate each of these terms using \eqref{eq:Bprop3}.  
  For $J_{4,1}^p$ we have
  \begin{equation}\label{eq:nonlinearFuckingAround1}
    \begin{split}
      |J_{4,1}^p| \leq&
      c_0 \theta( \| \bar{U}^n \|)  
      \| \bar{U}^n\|^{p-2}|A\bar{U}^n| \|U^n_*\| |A U^n_*|
      \leq  \frac{p}{8}\| \bar{U}^n\|^{p-2} |A\bar{U}^n|^2 
                    +  c \theta( \| \bar{U}^n \|)  \| \bar{U}^n\|^{p-2}
                        \|U^n_*\|^2 |A U^n_*|^2\\
      \leq&  \frac{p}{8}\| \bar{U}^n\|^{p-2} |A\bar{U}^n|^2 
                    + c\kappa^{p-2} \|U^n_*\|^2 |A U^n_*|^2.
   \end{split}
  \end{equation}
  For the next two terms we estimate
  \begin{equation}    \label{eq:fuckoffEstimatenum2}
   \begin{split}
     |J_{4,2}^p| + |J_{4,3}^p| \leq& c
    \theta( \| \bar{U}^n \|)  
     \| \bar{U}^n\|^{p-2} |A\bar{U}^n|^{3/2}
     \|U^n_*\|^{1/2} |A U^n_*|^{1/2} \| \bar{U}^n\|^{1/2}\\
     \leq& \frac{p}{8}\| \bar{U}^n\|^{p-2} |A\bar{U}^n|^{2}
                 + c  \theta( \| \bar{U}^n \|)
                 \| \bar{U}^n\|^{p}\|U^n_*\|^{2} |A U^n_*|^{2}\\
     \leq& \frac{p}{8}\| \bar{U}^n\|^{p-2} |A\bar{U}^n|^{2}
                 + c \kappa^p \|U^n_*\|^{2} |A U^n_*|^{2}.\\
  \end{split}
  \end{equation}
  The last term yields to the bounds
  \begin{equation}      \label{eq:ReallyFuckedCutoff2}
 \begin{split}
   |J_{4,4}^p| \leq c_0 p
   \theta( \| \bar{U}^n \|)  
    \| \bar{U}^n\|^{p-1} |A\bar{U}^n|^{2}
    \leq 2 \kappa p  c_0  \| \bar{U}^n\|^{p-2} |A\bar{U}^n|^{2}
    \leq \frac{p}{8} \| \bar{U}^n\|^{p-2} |A\bar{U}^n|^{2}.
 \end{split}
 \end{equation}
Note here that the last inequality follows from the requirement \eqref{eq:kappaCond}
imposed on $\kappa$.
 
 Finally we address the stochastic terms.  Observe that for any pair of
 stopping times $0 \leq \tau_a \leq \tau_b \leq T$,  the BDG inequality, 
 \eqref{eq:BDG} with $r = 1$, implies that
 \begin{equation}\label{eq:randomizedWanking}
  \begin{split}
     \E \sup_{\tau_a \leq  t \leq \tau_b} \left|\int_{\tau_a}^{t}
       J_5^{p} dW\right|
     &\leq c
     \E  \left( \int_{\tau_a}^{\tau_b}\| \bar{U}^n\|^{2(p-2)} 
                \langle \sigma^n (\bar{U}^n + U^n_*) 
                   , A \bar{U}^n \rangle ^2ds \right)^{1/2}\\
     &\leq c 
    \E \left( \int_{\tau_a}^{\tau_b}\| \bar{U}^n\|^{2(p-1)} 
           (1  + \|U^n_*\| + \|\bar{U}^n\|)^2 ds
                  \right)^{1/2}\\
          &\leq c \left(
    \E \sup_{\tau_a \leq  t \leq \tau_b}  \| \bar{U}^n\|^{p-1} 
   \left( \int_{\tau_a}^{\tau_b}
          (1  + \|U^n_*\| + \|\bar{U}^n\|)^2 ds
                 \right)^{1/2} \right) \\ 
            &\leq 
   \frac{1}{2} 
   \E \left( 
     \sup_{\tau_a \leq  t \leq \tau_b}  \|\bar{U}^n\|^{p} 
     \right)
     +
   c\E\left( \int_{\tau_a}^{\tau_b}
         (1  + \|U^n_*\| + \|\bar{U}^n\|)^2 ds
                \right)^{p/2}\\ 
            &\leq 
   \frac{1}{2} 
   \E \left( 
     \sup_{\tau_a \leq  t \leq \tau_b}  \|\bar{U}^n\|^{p} 
     \right)
     + c\E
    \int_{\tau_a}^{\tau_b}
         ((1  + \|U^n_*\|)^p + \|\bar{U}^n\|^p) ds.
                \\ 
   \end{split}
 \end{equation}
Combining the estimates \eqref{eq:sublinearEst1},
\eqref{eq:easyEstJ3J4}, \eqref{eq:nonlinearBspliting},
\eqref{eq:nonlinearFuckingAround1}, \eqref{eq:fuckoffEstimatenum2},
\eqref{eq:ReallyFuckedCutoff2}, \eqref{eq:randomizedWanking}  we find, 
for any $t \in (0,T]$,
\begin{equation}
  \begin{split}
     \E  \biggl( 
       \sup_{t' \in [0, t]} \|\bar{U}^n \|^p &+ 
      \int_{0}^{t}  |A \bar{U}^n|^2 \|\bar{U}^n \|^{p-2} dt'
      \biggr)
      \leq
      c
      \E \int_{0}^{t} \left(
      	\|\bar{U}^{n}\|^{p}  + (1 + |AU^n_*|^{2} \|U^n_*\|^{2}  + \| U^n_*\|^p )
	\right) dt'\\
	      \leq&
      c
      \int_{0}^{t} \left( \E  \sup_{s \in [0, t']}
      	\|\bar{U}^{n}\|^{p}  +  \E (1 + |AU^n_*|^{2} \|U^n_*\|^{2}  + \| U^n_*\|^p )
	\right) dt'.\\
  \end{split}
\end{equation}
Applying then the Gronwall
inequality yields
\begin{equation}\label{eq:FinalConcluUbar}
  \begin{split}
    \E  \left( 
       \sup_{ 0 \leq t' \leq T} \|\bar{U}^n \|^p + 
      \int_0^T |A \bar{U}^n|^2 \|\bar{U}^n \|^{p-2} dt'
      \right)
      &\leq c
      \E \int_0^T (1 + |AU^n_*|^{2} \|U^n_*\|^{2}  + \| U^n_*\|^p )dt'\\
      & \leq c \E (1 + \|U_0\|)^{\max\{ p, 4\}}.
  \end{split}
\end{equation}
The second inequality follows from \eqref{eq:LpTimeEstimatesUstar}.
We also note that the term involving $|AU^n_*|^2\|U^n_*\|^2$ is
responsible for the first part of the moment condition
\eqref{eq:DataMomentConditionMartingale}.

In order to complete the proof of (\ref{eq:UniformGalEst1}) we observe that
\begin{equation}\label{eq:UnifromGalEstMoreProblems}
  \begin{split}
   \E  \biggl(&
      \sup_{0 \leq t' \leq T} \|U^n \|^p +
     \int_0^T |A U^n|^2 \|U^n \|^{p-2} dt'
     \biggr) 
     \leq c
     \E \left( \sup_{0 \leq t' \leq T} \|U^n \|^p 
  + \left(\int_0^T |A U^n|^2dt' \right)^{p/2} \right)\\
    \leq& c
     \E \left( \sup_{0 \leq t' \leq T} \|U^n_* \|^p 
   + \left(\int_0^T |A U^n_*|^2dt' \right)^{p/2} \right)
    +c \E \left( \sup_{0 \leq t' \leq T} \|\bar{U}^n \|^p 
  + \left(\int_0^T |A \bar{U}^n|^2dt' \right)^{p/2} \right).\\
\end{split}
\end{equation}
Given the estimates \eqref{eq:FinalConcluUbar} 
for $\bar{U}^{n}$ and \eqref{eq:LpTimeEstimatesUstar} for $U_{*}^{n}$ it therefore remains
to estimate the last term, i.e. to prove the analogue of (\ref{eq:UniformGalEst2}) for $\bar{U}$.  
Returning to (\ref{eq:ItoEvolutoinShiftedEqn}) for the case $p =2$ we must therefore find suitable estimates
for the left hand side of the expression
\begin{equation}\label{eq:returntoOZahh}
  \begin{split}
   \E \biggl(\int_0^T |A \bar{U}|^2 dt' \biggr)^{p/2}
       \leq&
       \E \left( \int_0^T |J_1^2|  
       	+ |J_2^2|  + |J_{2}^{4}|ds
         + \sup_{t \in [0,T]} \left| \int_0^t J_5^2 dW
       \right| \right)^{p/2}.
    \end{split}
\end{equation}
Note that, when $p =2$, $|J_3^p| =  0$.  By treating $|J_{1}^{2}|$, $|J_{2}^{2}|$
in a similar manner to (\ref{eq:sublinearEst1}), (\ref{eq:easyEstJ3J4}),
we infer
\begin{equation}\label{eq:returntoOZahhI1}
  \begin{split}
  		|J_{1}^{2}| + |J_{2}^{2}| 
                 \leq 2^{-(2 + 2/p)}|A \bar{U}^{n}|^{2}
                 	         + c (1+ \|U_{*}^{n}\|^{2}) 
	         	        +c\|\bar{U}^{n}\|^{2}.
  \end{split}
\end{equation} 
For $|J_{4}^{2}|$ we 
estimate similarly to \eqref{eq:nonlinearBspliting}, \eqref{eq:nonlinearFuckingAround1},
\eqref{eq:fuckoffEstimatenum2}, \eqref{eq:ReallyFuckedCutoff2}
to deduce
\begin{equation}\label{eq:nonlinearEstMods}
	|J_{4}^{2}|
		\leq 2^{-(3+2/p) } |A \bar{U}^{n}|^{2}
		         + c \|U_{*}^{n}\|^{2}|AU_{*}^{n}|^{2}
		         + 4 \kappa c_0 |A\bar{U}|^{2} 
                 \leq 2^{-(2 + 2/p)}|A \bar{U}^{n}|^{2}
                 	         + c \|U_{*}^{n}\|^{2}|AU_{*}^{n}|^{2}.
\end{equation}
The constant $c_0$ after the first inequality is
from \eqref{eq:Bprop3}.  
Thus, the assumption \eqref{eq:kappaCond} justifies the second inequality.
For the stochastic intergral term in \eqref{eq:returntoOZahh} 
we apply the BDG inequality, \eqref{eq:BDG}, and deduce:
\begin{equation}\label{eq:BDGp2InInEqul}
  \begin{split}
    \E \sup_{t \in [0,T]} &\left| \int_0^t J_5^2 dW
       \right|^{p/2} \\
       \leq& c\E \left(\int_0^T  \langle \sigma^n (\bar{U}^n + U^n_*) 
                   , A \bar{U}^n \rangle ^2 dt'\right)^{p/4}
        \leq c\E \left(\int_0^T  (1 + \|\bar{U}^n\|^2 + \|U^n_*\|^2)
                  \|\bar{U}^n\|^2  dt'\right)^{p/4}\\
        \leq& c\E \left(\int_0^T  (1 + \|\bar{U}^n\|^4 + \|U^n_*\|^4)dt'\right)^{p/4}
        \leq c\E \int_0^T  (1 + \|\bar{U}^n\|^p + \|U^n_*\|^p)dt'.
 \end{split}
\end{equation}
Applying \eqref{eq:returntoOZahhI1}, \eqref{eq:nonlinearEstMods}
and \eqref{eq:BDGp2InInEqul}  to \eqref{eq:returntoOZahh} we have
\begin{equation}\label{eq:returntoOZahhIConclusion}
  \begin{split}
    \E \biggl(\int_0^T& |A \bar{U}|^2 dt' \biggr)^{p/2}\\
    \leq& \frac{1}{2} \E
     \left( \int_0^T |A\bar{U}|^2 ds \right)^{p/2} \! \! \!
        +
       c \E  \int_0^T  (1+ \|U_*^n\|^p  +  \| \bar{U}^n \|^p) dt'
       +c\E \left( \int_{0}^{T} \|U^n_*\|^2 |AU^n_*|^2 ds \right)^{p/2}\\
    \leq& \frac{1}{2} \E
     \left( \int_0^T |A\bar{U}|^2 ds \right)^{p/2} 
        \! +
       c \E \sup_{t \in [0,T]} (1 + \|U_*^n\|^p  +  \| \bar{U}^n \|^p)
       + \E \|U_0\|^{2p}.
\end{split}
\end{equation}
Note that the terms involving
$\|U^n_*\|^2|AU^n_*|^2$ are treated in the final
inequality using \eqref{eq:LpTimeEstimatesUstar} and are responsible
for the second part of the moment condition \eqref{eq:DataMomentConditionMartingale}.
Applying \eqref{eq:returntoOZahhIConclusion} in turn to
\eqref{eq:UnifromGalEstMoreProblems} we finally conclude
\eqref{eq:UniformGalEst1}.  With \eqref{eq:LpTimeEstimatesUstar},
\eqref{eq:UniformGalEst2} also now follows from 
\eqref{eq:returntoOZahhIConclusion}

The bound~\eqref{eq:uniformFracTmEstGalerkin} is a
direct application of \eqref{eq:BDGfrac} with \eqref{eq:BndCondSig}:
\begin{equation}\label{eq:BDGfrackapplication}
  \begin{split}
   \E \left(  
   \left| \int_0^{t} \sigma^n(U^n) dW \right|_{W^{\alpha, p}([0,T];H)}^p
      \right)
       \leq&c
         \E
       \int_0^{T} |\sigma^n(U^n)|_{L_2(\mathfrak{U}, H)}^pdt 
       \leq c
         \E 
       \int_0^{T} (1 + |U^n|^p) dt.\\
\end{split}
\end{equation}
We finally establish~\eqref{eq:UniformDrftyEst}.
Integrating (\ref{eq:GalerkinCutoff}) we observe that
\begin{equation}\label{eq:IntegratedMinusNoise}
  \begin{split}
  U^n&(t) - \int_0^t\sigma^n(U^n) dW  
  =   U^n_0 +
    \int_0^t [A U^n + \theta( \| U^n - U^n_*  \| ) B^n(U^n) + F^n(U^n) ]dt.
  \end{split}
\end{equation}
With,  (\ref{eq:sublinerConF}), (\ref{eq:Bprop3}) we infer:
\begin{equation}  \label{eq:W12NonStochasticEst}
  \begin{split}
   \biggl| U^n(t) - \int_0^t\sigma^n(U^n) dW \biggr|_{W^{1, 2}([0,T];H)}^2
    \leq&
    c |U_0|^2 + c\int_0^T (|AU^n|^2 + |B^n(U^n)|^2 + |F^n(U^n)|^2)ds\\
    \leq&
    c |U_0|^2 + c\int_0^T (|AU^n|^2 + |B(U^n)|^2 + |F(U^n)|^2)ds\\
    \leq&
    c |U_0|^2 + c\int_0^T (1 + \|U^n\|^2)(1 +|AU^n|^2) ds.
  \end{split}
\end{equation}
Taking expected values in this expression and applying \eqref{eq:UniformGalEst1} (i) for
the case $p =4$ gives \eqref{eq:UniformDrftyEst}.  The
proof is now complete.
\end{proof}

\section{Local Existence of Martingale Solutions}
\label{sec:local-existence}

In this section we establish the existence of a Martingale solution of
\eqref{eq:AbsEq}.  The first step is to make use of the uniform estimates
established in Lemma~\ref{thm:UniformEstimateGalerkin} we infer
the compactness (in certain spaces) of the probability laws associated
to the Galerkin approximations.  We then change the underlying
probabilistic basis in order to find a new sequence of random elements
equal in law to the original Galerkin approximations but which
converge almost surely. The technical details of the passage 
to the limit, which is used also below for the case of pathwise solutions,
is carried out in Section~\ref{sec:passToLim} below.

\subsection{Compactness Arguments}
\label{sec:compArgs}

For a given initial distribution $\mu_0$ on $V$ we fix a stochastic basis 
$\mathcal{S} = (\Omega, \mathcal{F},$ $\{\mathcal{F}_t\}_{t \geq
  0}, \mathbb{P}, W)$ upon which is defined an $\mathcal{F}_0$ measurable
random element $U_0$ with distribution $\mu_0$.  Consider the sequence
of Galerkin approximations $\{U^n\}$ solving \eqref{eq:GalerkinCutoff}
relative to this basis and initial condition. We consider the phase
space:
\begin{equation}\label{eq:phaseSpaceStrongConv}
  \begin{split}
  \mathcal{X}_U = (L^2(0,T; V) \cap C([0,T], V'), \quad
  \mathcal{X}_W = C([0,T],\mathfrak{U}_0), \quad
     \mathcal{X} &= \mathcal{X}_U \times \mathcal{X}_W.
  \end{split}
\end{equation}
We may think of the first component, $\mathcal{X}_U$, of this phase space as 
the set where the
solution $U^n$ lives and the second component, $\mathcal{X}_W$,
as being the set on which the driving Brownian motions are defined. We
consider the probability measures 
\begin{equation}\label{eq:solnMeasures}
  \mu^n_{U}(\cdot) = \mathbb{P}( U^n \in \cdot ) \in Pr(  L^2([0,T]; V) \cap C([0,T],V')),
\end{equation}
and
\begin{equation}\label{eq:BMMeasures}
  \mu_W(\cdot) = \mu^n_W(\cdot) = \mathbb{P}(W \in \cdot) \in Pr( C([0,T], \mathfrak{U}_0)).
\end{equation}
This defines a sequence of probability measures
$\mu^n =
\mu^n_{U} \times \mu^n_{W}$ on the phase space $\mathcal{X}$.  By
making appropiate use of Lemma~\ref{thm:UniformEstimateGalerkin} we will
now show that this sequence is tight.  More precisely:
\begin{Lem}\label{thm:tightness1pt0}
  Suppose that $\mu_{0}$ satisfies \eqref{eq:U0CondMartingale} with $q \geq 8$. Consider 
  the measures $\mu^n$ on $\mathcal{X}$ defined according to \eqref{eq:solnMeasures}, 
  \eqref{eq:BMMeasures}. Then the sequence $\{\mu^{n}\}_{n \geq 1}$ is tight and therefore weakly
  compact over the phase space $\mathcal{X}$.
\end{Lem}

\begin{proof}
By applying Lemma~\ref{thm:FlanGatComp}, (i) with $X_{-1} = H$, $X_0 = V$,
$X_1 = D(A)$, $p = 2$ and $\alpha = 1/4$ we deduce that
\begin{displaymath}
  L^2([0,T]; D(A)) \cap W^{1/4,2}([0,T]; H)       
       \subset \subset L^2([0,T];V).
\end{displaymath}
For $R > 0$ we define the set
\begin{displaymath}
\begin{split}
  B_R^1 = \{U \in  L^2([0,T]; D(A)) &\cap W^{1/4,2}([0,T]; H) : 
  |U|^2_{L^2([0,T]; D(A))}  + |U|^2_{W^{1/4,2}([0,T]; H)} \leq R^2\}
\end{split}
\end{displaymath}
which is thus compact in $L^2([0,T],V)$.  Due to the Chebyshev
inequality and the uniform estimates (\ref{eq:UniformGalEst1}), \eqref{eq:UniformDrftyEst},
(\ref{eq:uniformFracTmEstGalerkin}) in the case $p = 2$, we estimate,
\begin{equation}\label{eq:ChebBndTight1}
  \begin{split}
    \mu^n_{U}( (B_R^1)^C) 
    =& \Prb( |U^n|^2_{L^2([0,T]; D(A))}  + |U^n|^2_{W^{1/4,2}([0,T];
     H)} \geq R^2)\\
    \leq& \Prb( |U^n|^2_{L^2([0,T]; D(A))}  \geq R^2/2) +
           \Prb(|U^n|^2_{W^{1/4,2}([0,T];H)} \geq R^2/2)\\
    \leq& \frac{2}{R^2}
    \E\left( \int_0^T |AU^n|^2 dt'  + \left| U^n
      \right|_{W^{\frac{1}{4},2}([0,T];H)}^2 \right)
    \leq \frac{c}{R^2},
  \end{split}
\end{equation}
where the numerical constant $c$ is independent of $n$.

Choose $\alpha \in (1/q, 1/2)$ so that $\alpha q > 1$.  By Lemma~\ref{thm:FlanGatComp}, (ii) with
$Y_0 = V' = D(A^{-1/2})$ and $Y
= H$ we infer the compact embeddings 
\begin{displaymath}
  W^{1, 2}([0,T]; H) \subset \subset C([0,T],V'), \quad
  W^{\alpha, q}([0,T]; H) \subset \subset C([0,T],V').
\end{displaymath}
For $R > 0$, we take $B_R^{2,1}$ and $B_R^{2,2}$ to be the balls of
radius $R$ in $W^{1, 2}([0,T],H)$ and $W^{\alpha, q}([0,T], H)$
respectively.  It follows that for  $R > 0$, $B_R^2 := B_R^{2,1} +
B_R^{2,2}$ is compact in $C([0,T],V')$.  Since indeed,
\begin{displaymath}
   \{ U^{n} \in B_{R}^{2} \} \supset
   \left\{ U^n(t) - \int_0^t \sigma^n (U^n)dW \in B_R^{2,1} \right\} \cap 
   \left\{   \int_0^{t} \sigma^n(U^n) dW \in B_R^{2,2}\right\},
\end{displaymath}
we may apply Chebyshev's inequality and then the
uniform estimates \eqref{eq:UniformDrftyEst}
\eqref{eq:uniformFracTmEstGalerkin} to infer
\begin{equation}\label{eq:ChebBndTight2}
  \begin{split}
    \mu^n_{U}( (B_R^2)^C) 
       \leq& 
       \Prb \left( 
         \left|U^n(t) - \int_0^t \sigma^n (U^n)dW
         \right|_{W^{1,2}([0,T];H)}^2 
         \geq R^2 \right)\\
       &+ 
       \Prb \left(  \left| \int_0^{t} \sigma^n(U^n) dW
         \right|_{W^{\alpha, q}([0,T];H)}^q
         \geq R^q \right)\\
       \leq& \frac{c}{R^2}.
  \end{split}
\end{equation}
As above the $c$ is independent of $n$.

It is not hard to see\footnote{One need only verify that if
  $\{U^n\}_{n\geq 0} \subset L^2(0,T; V) \cap C([0,T], V')$ and if
   $U^n \rightarrow U$  in $L^2(0,T; V)$,
    $U^n \rightarrow \tilde{U}$  in  $C([0,T], V')$
that $U = \tilde{U}$} 
that $B^1_R \cap B^2_R$ is compact in
$L^2(0,T; V) \cap C([0,T], V')$ for every $R > 0$.   As a consequence
of (\ref{eq:ChebBndTight1}) and (\ref{eq:ChebBndTight2}) we have
\begin{displaymath}
  \mu^n_{U}((B_R^1 \cap B_R^2)^C) \leq \mu^n_{U}((B_R^1)^C) + \mu^n_{U}((B_R^2)^C)
  \leq \frac{c}{R^2}
\end{displaymath}
We therefore take $A_\epsilon := B^1_{\sqrt{2c/\epsilon}} \cap
B^2_{\sqrt{2c/\epsilon}}$, with $c$ the constant which appears on the left
hand side immediately
above.  With this definition we infer that for $\epsilon > 0$,
\begin{equation}\label{eq:ChebFinal1}
  \mu^n_{U}(A_\epsilon) \geq 1-\frac{\epsilon}{2},
\end{equation}
over all $n$.

We next turn to the sequence $\{\mu^n_W\}_{n\geq 0}$.  This sequence is
constantly equal to one element and is thus
weakly compact.  Hence, as a consequence of \physD{Prohorov's Theorem (see e.g. \cite{ZabczykDaPrato1}),}\arxiv{Proposition~\ref{thm:ProhorovSkorhod}, (i)} $\{\mu^n_W\}_{n\geq 0}$ must be
tight.  We therefore infer the existence of collection of compact sets
$\tilde{A}_\epsilon \subset C([0,T], \mathfrak{U}_0)$ so that
\begin{equation}\label{eq:ChebFinal2}
  \mu^n_W(\tilde{A}_\epsilon) \geq 1 - \frac{\epsilon}{2}    
\end{equation}
for all $n$.

We now have everything in hand to conclude the tightness and therefore
the weak compactness of $\{\mu^n\}_{n \geq 0}$.  For $\epsilon > 0$ we
define $\mathcal{K}_\epsilon := A_\epsilon \times \tilde{A}_\epsilon$
which are compact in $\mathcal{X}$.  By (\ref{eq:ChebFinal1}) and
(\ref{eq:ChebFinal2}) we infer that, for any $\epsilon > 0$ and every
$n$, 
\begin{displaymath}
  \mu^n(\mathcal{K}_\epsilon) \geq 1 -\epsilon
\end{displaymath}
and thus that $\{\mu^n\}_{n \geq 0}$ is tight in $\mathcal{X}$.    Prohorov's theorem\arxiv{, given herein
as Proposition~\ref{thm:ProhorovSkorhod}} therefore implies that
$\mu^n$ is weakly compact.  The proof is therefore complete.
\end{proof}

\subsubsection{Strong Convergence on the Skorohod Space}
\label{sec:strong-conv-skor}

Given $\mu_0$ (satisfying \eqref{eq:U0CondMartingale} with $q \geq 8$) we
have shown that the sequence of measures $\{\mu^n\}_{n \geq 1}$ associated to the Galerkin sequence
$(U^n,W)$ are weakly compact on $\mathcal{X}$.  Passing to a weakly convergent 
subsequence $\mu^{n_k}$ we now 
apply the Skorohod embedding theorem, 
\arxiv{Proposition~\ref{thm:ProhorovSkorhod},}\physD{see e.g. \cite{ZabczykDaPrato1},}
 to infer the following Proposition.
\begin{Prop}\label{thm:SubsequentialConv}
  Suppose that $\mu_0$ is a probability measure on $V$ 
  satisfying \eqref{eq:U0CondMartingale} with $p > 4$.   Then there
  exists a probability space $(\tilde{\Omega}, \tilde{\mathcal{F}},
  \tilde{\Prb})$ and a
  subsequence $n_k$ and a sequence of $\mathcal{X}$ valued random
  variables $(\tilde{U}^{n_k}, \tilde{W}^{n_k})$ such that
  \begin{itemize}
  \item[(i)] $(\tilde{U}^{n_k}, \tilde{W}^{n_k})$ converges almost
    surely, in the topology of $\mathcal{X}$, to an element $(\tilde{U}, \tilde{W})$.
  \item[(ii)] $\tilde{W}^{n_k}$ is a cylindrical Wiener process,
    relative to the filtration $\mathcal{F}_t^{m_k}$, given by the
    completion of $\sigma(\tilde{W}^{m_k}(s),$ $\tilde{U}^{m_k}(s); s
    \leq t)$.
\item[(iii)] Each pair $(\tilde{U}^{n_k}, \tilde{W}^{n_k})$ satisfies
    \begin{equation}\label{eq:GalerkinCutoffNewBasis}
      \begin{split}
        d \tilde{U}^{n_k} + [A \tilde{U}^{n_k} + \theta( \| \tilde{U}^{n_k} - \tilde{U}_*^{n_k} \| ) 
           &B^{n_k}(\tilde{U}^{n_k}) + F^{n_k}(\tilde{U}^{n_k}) ]dt
           =\sigma^{n_k}(\tilde{U}^{n_k}) d\tilde{W}^{n_k},\\
           \tilde{U}^{n_k}(0) &= P_{n_k} \tilde{U}(0)^{n_k} := \tilde{U}^{n_k}_0,
      \end{split}
    \end{equation}
    where we define $\tilde{U}_*^{n_k}$ by:
    \begin{equation}\label{eq:AuxLinearmodNewBasis}
      \frac{d}{dt} \tilde{U}_*^{n_k} + A \tilde{U}_*^{n_k} = 0 
      \quad \tilde{U}_*^{n_k}(0) = \tilde{U}^{n_k}_0.
    \end{equation}
  \end{itemize}
\end{Prop}

With this proposition established the existence of a local
Martingale solution follows once 
we have shown that $(\tilde{U}, \tilde{W})$ and an
appropriately defined stopping time $\tau$ 
(see \eqref{eq:firstExistenceTime}) satisfy 
\eqref{eq:AbsEq}.  This passage to the limit 
argument, which is technical and delicate, is carried out
in Section~\ref{sec:passToLim} below.

\begin{Rmk}\label{rmk:ProofStrongConv}
While Proposition~\ref{thm:SubsequentialConv}, (i)
follows directly from \arxiv{Proposition~\ref{thm:ProhorovSkorhod}, (ii)}\physD{Skorohod's theorem (see, for example, \cite{ZabczykDaPrato1}),}
further steps are required
to establish (ii), (iii).  These technical points may be demonstrated
in a similar manner to previous works.  See \cite{Bensoussan1}.
\end{Rmk}

\section{Local Pathwise Solutions}
\label{sec:localmax-pathwisesolns}

We turn now to study Pathwise solutions of \eqref{eq:AbsEq}.    Here the 
key step is to apply Proposition~\ref{thm:GyongyKry} in order to show that 
$(U^{n}, W)$ converges almost surely in $L^{2}([0,T]; V) \cap C([0,T], V')$ 
relative to the initial stochastic basis.  The diagonal condition,
\eqref{eq:diagonalCondGK} translates to a question of pathwise uniqueness
which we address first.

\subsection{Local Pathwise Uniqueness}
\label{sec:local-pathw-uniq}

The following proposition establishes the uniqueness, pathwise, for any pair of
solutions of the modified system \eqref{eq:modFullSystemCut}.   Such
solutions appear in an intermediate step in the compactness arguments 
in Section~\ref{sec:compRevist} below.
\begin{Prop}\label{thm:Uniqueness}
  Suppose that $(\mathcal{S}, U^{(1)})$ and $(\mathcal{S}, U^{(2)})$ are two global
  Martingale solutions of \eqref{eq:modFullSystemCut} relative to the same stochastic basis
  $\mathcal{S} := (\Omega, \mathcal{F}, \{\mathcal{F}_t\}_{t \geq
    0},$ $\mathbb{P}, W)$.  Assume that, in addition to the conditions imposed
  in Definition~\ref{def:MartingaleSoln}, $F$ and $\sigma$ satisfy the
  Lipschitz conditions (\ref{eq:lipConF}) and (\ref{eq:lipCondSig}). Define
  \begin{equation}\label{eq:InitialDataEqSet}
    \Omega_0 = \{U^{(1)}(0) = U^{(2)}(0) \}.
  \end{equation}
  Then $U^{(1)}$ and $U^{(2)}$ are indistinguishable on $\Omega_0$ 
  in the sense that
   \begin{equation}\label{eq:uniquenessCutoffSolns}
    \Prb \left( \indFn{\Omega_0} ( U^1(t) - U^2(t)  = 0; \forall t
      \geq 0 \right) = 1.
  \end{equation}
\end{Prop}
\begin{Rmk}\label{rmk:pathwiseUniqueness}
  We note that, with trivial modifications to the proof that follows,
  one may establish that Pathwise solutions of \eqref{eq:AbsEq} are
  unique in the sense of Definition~\ref{def:Pathwisesolns}, (ii).
\end{Rmk}
\begin{proof}[Proof of Proposition~\ref{thm:Uniqueness}]
  Define $R = U^{(1)} - U^{(2)}$ and let $\bar{R} =
  \indFn{\Omega_0}R$.  Note that, by definition,
  $\bar{R} \in
  C([0,\infty); V) \cap L^2_{loc}([0,\infty); D(A))$, a.s.
  Due to the bilinear term $B$, when we attempt to estimate
  $\bar{R}$, stray terms arise that involve only $U^{(1)}$ or
  $U^{(2)}$.   See \eqref{eq:cuttTermsLipUse}, \eqref{eq:UniquePathEstJ1} 
  below. To remedy this situation we define the stopping times
 \begin{equation}\label{eq:pointwiseControl}
   \tau^{(n)}   := \inf_{t \geq 0} \left\{ \int_0^{t} \|U^{(1)}\|^2
   |AU^{(1)}|^2 + \|U^{(2)}\|^2 |AU^{(2)}|^2ds  \geq n \right\}.
 \end{equation}
 Clearly this is an increasing sequence. 
 Futhermore, since $U^{(1)}$, $U^{(2)}$ are global solutions,
 we may infer that $lim_{n \rightarrow \infty} \tau^{(n)} =\infty$
 from \eqref{eq:regularityCondMG}.
 Hence, the desired result will follow if we show 
 that for any $n$,  $T > 0$,
\begin{equation}\label{eq:suffConForUniqueness}
  \E \left( \sup_{[0,\tau^{(n)} \wedge T]}\|\bar{R}\|^2 \right)  =0.
\end{equation}

Subtracting the equations (c.f. \eqref{eq:modFullSystemCut}) for $U^{(2)}$ from that
for $U^{(1)}$ we arrive at the following equation for $R$:
\begin{equation}\label{eq:EvolutionOfDifferenceOfSoln}
  \begin{split}
  d R + (AR &+ \theta( \| U^{(1)} - U^{(1)}_*\| ) B(U^{(1)}) -
                        \theta( \| U^{(2)} - U^{(2)}_*\| ) B(U^{(2)}) \\
      &+ F(U^{(1)}) - F(U^{(2)}))dt
         = (\sigma(U^{(1)}) - \sigma(U^{(2)}))dW,\\
     R(0) &= U^{(1)}(0) - U^{(2)}(0).
  \end{split}
\end{equation}
It\={o}'s lemma yields the following evolution equation for $\|R\|^2$:
\begin{equation}\label{eq:RHTypeNorm}
  \begin{split}
    d \|R\|^2 + 2|AR|^2
    =& 2\langle
       \theta( \| U^{(2)} - U^{(2)}_*\| ) B(U^{(2)}) 
       - \theta( \| U^{(1)} - U^{(1)}_*\| )  B(U^{(1)}), AR \rangle\\
      &+ 2\langle F(U^{(2)}) - F(U^{(1)}), AR\rangle dt
        + \|\sigma(U^{(1)}) - \sigma(U^{(2)})\|_{L_2(\mathfrak{U}, V )}^2 dt\\
      & + 2 \langle \sigma(U^{(1)}) - \sigma(U^{(2)}), AR \rangle dW.\\
 \end{split}
\end{equation}
Fix $n$ and stopping times $\tau_a, \tau_b$, such that $0 \leq \tau_a \leq \tau_b \leq \tau^{(n)}$. Integrating in time and
taking supremums, multiplying by $\indFn{\Omega_0}$ and finally taking
an expected value we arrive at the expression
\begin{equation}\label{eq:supTforRV1}
  \begin{split}
    \E \biggl(&\sup_{t \in [\tau_a, \tau_b]} \|\bar{R}\|^2 + \int_{\tau_a}^{\tau_b} |A \bar{R}|^2ds
       \biggr)\\
    \leq& \E \|\bar{R}(\tau_a)\|^2
       + 2 \E \int_{\tau_a}^{\tau_b} |\langle 
       (\theta( \| U^{(1)} - U^{(1)}_*\| ) - 
       \theta( \| U^{(2)} -U^{(2)}_*\| )) B(U^{(1)}),  A \bar{R}\rangle| dt\\
       &\quad + 2 \E \int_{\tau_a}^{\tau_b} |\langle 
      B(U^{(1)}) - 
       B(U^{(2)}) , A \bar{R}\rangle| dt
      +2 \E \int_{\tau_a}^{\tau_b} |\langle F(U^{(1)}) - F(U^{(2)}), A \bar{R}\rangle| dt\\
      &\quad +2 \E \sup_{t \in [\tau_a, \tau_b]}  \left|
       \int_{\tau_a}^{t}\langle \sigma(U^{(1)}) - \sigma(U^{(2)}), A \bar{R} \rangle dW
       \right|
        +\E \int_{\tau_a}^{\tau_b}\indFn{\Omega_0} 
      \|\sigma(U^{(1)}) - 
          \sigma(U^{(2)})\|_{L_2(\mathfrak{U}, V)}^2ds\\
  :=& \E \|\bar{R}(\tau_a)\|^2 + J_1 +  J_2 + J_3 + J_4 + J_5.
  \end{split}
\end{equation}
For $J_1$, we use that $\theta$ is Lipschitz.  See \eqref{eq:CuttoffDef}.
Applying (\ref{eq:Bprop3}), we have
\begin{equation}\label{eq:cuttTermsLipUse}
  \begin{split}
    J_1 \leq& 
    c \E \int_{\tau_a}^{\tau_b} 
    \| (U^{(1)} - U^{(2)})  -  (U_*^{(1)} - U_*^{(2)})  \| 
   |\langle B(U^{(1)}),  A \bar{R}\rangle| dt
    \leq  c \E \int_{\tau_a}^{\tau_b}     \| \bar{R} \|
        \|U^{(1)} \| | AU^{(1)}| |A \bar{R}| dt\\   
     \leq& \frac{1}{4} \E \int_{\tau_a}^{\tau_b}   |A\bar{R} |^{2} ds
           + c
   \E \int_{\tau_a}^{\tau_b} \|U^{(1)}\|^{2} |AU^{(1)} |^{2}  \|\bar{R}\|^{2} ds. 
  \end{split}
\end{equation}
Note that since both $U_*^{(1)}$, $U_*^{(2)}$ satisfy the linear
equation \eqref{eq:linearAuxSystem} it is clear that, for every $t \geq 0$
$\indFn{\Omega_0} (U_*^{(1)}(t) - U_*^{(2)}(t)) = 0$ almost surely.
For $J_2$ the bilinearity of $B$ and \eqref{eq:Bprop3} imply:
\begin{equation}\label{eq:UniquePathEstJ1}
  \begin{split}
    J_2 = 2\E &\int_{\tau_a}^{\tau_b} |\langle 
   B(U^{(1)} - U^{(2)}, U^{(1)})
  + B(U^{(2)}, U^{(1)} - U^{(2)}) , A \bar{R}\rangle| dt\\
   \leq&
   2\E \int_{\tau_a}^{\tau_b}  |\langle 
    B(\bar{R}, U^{(1)})
  + B(U^{(2)}, \bar{R}) , A \bar{R}\rangle| dt\\
     \leq& c
  \E \int_{\tau_a}^{\tau_b}    
  (\|U^{(1)}\|^{1/2} |AU^{(1)} |^{1/2}  +  \|U^{(2)}\|^{1/2} |AU^{(2)} |^{1/2})
  \|\bar{R}\|^{1/2} |A\bar{R} |^{3/2}ds\\
     \leq& \frac{1}{4} \E \int_{\tau_a}^{\tau_b}   |A\bar{R} |^{2} ds
           + c
   \E \int_{\tau_a}^{\tau_b} (\|U^{(1)}\|^{2} |AU^{(1)} |^{2}  +  \|U^{(2)}\|^{2} |AU^{(2)} |^{2})
  \|\bar{R}\|^{2} ds.
  \end{split}
\end{equation}
The terms $J_3$ and $J_5$ are estimated directly making use of
(\ref{eq:lipConF}) to infer
\begin{equation}\label{eq:FLipUniqueEst}
  \begin{split}
    J_3 \leq& c
    \E \int_{\tau_a}^{\tau_b} \indFn{\Omega_0} 
        |F(U^{(1)}) - F(U^{(2)}) | |A \bar{R}|ds\\
        \leq&  c
        \E \int_{\tau_a}^{\tau_b} 
       \|\bar{R}\|
       |A \bar{R}|ds
        \leq
        \frac{1}{4} \E \int_{\tau_a}^{\tau_b}        |A \bar{R} |^2ds
        +c\E \int_{\tau_a}^{\tau_b} \|\bar{R}\|^2ds,
   \end{split}
\end{equation}
and making use of (\ref{eq:lipCondSig}) to deduce
\begin{equation}  \label{eq:GLipUniqueEst}
 \begin{split}
   J_5 \leq c \E \int_{\tau_a}^{\tau_b} \|\bar{R}\|^2ds.
  \end{split}
\end{equation}
Finally, $J_4$ is addressed using (\ref{eq:BDG}), with $r =1$ and then 
 (\ref{eq:lipCondSig})
\begin{equation}\label{eq:J3StochasticUniquenessThm}
  \begin{split}
  J_4 \leq& 
    c\E  \left(
  \int_{\tau_a}^{\tau_b} \indFn{\Omega_0} \langle \sigma(U^{(1)}) -
  \sigma(U^{(2)}),
   A \bar{R} \rangle^{2} ds
  \right)^{1/2}\\
\leq& c \E  \left(
  \int_{\tau_a}^{\tau_b} \indFn{\Omega_0}\|\sigma(U^{(1)}) -
  \sigma(U^{(2)})\|^2_{L_2(\mathfrak{U},V)} 
  \| \bar{R}\|^2 ds
  \right)^{1/2}\\
\leq& c
\E  \left(
 \int_{\tau_a}^{\tau_b} \|\bar{R}\|^4
 \right)^{1/2}
\leq \frac{1}{2} \E  
 \sup_{t \in [\tau_a, \tau_b]} \|\bar{R}\|^2
+c \E  \int_{\tau_a}^{\tau_b} \| \bar{R}\|^2 ds.
\end{split}
\end{equation}
Applying the estimates in \eqref{eq:cuttTermsLipUse}, \eqref{eq:UniquePathEstJ1}, \eqref{eq:FLipUniqueEst},
\eqref{eq:GLipUniqueEst}, \eqref{eq:J3StochasticUniquenessThm} to \eqref{eq:supTforRV1}
we infer
\begin{equation}\label{eq:FinalConcForStochGronUnique}
  \begin{split}
    \E \biggl(\sup_{t \in [\tau_a, \tau_b]} &\|\bar{R}\|^2 + \int_{\tau_a}^{\tau_b} |A \bar{R}|^2ds
       \biggr)\\
    \leq& c\E \|\bar{R}(\tau_a)\|^2 
     + c \E \int_{\tau_a}^{\tau_b} (\|U^{(1)}\|^{2} |AU^{(1)} |^{2}  +  \|U^{(2)}\|^{2} |AU^{(2)} |^{2} + 1)
  \|\bar{R}\|^{2} ds.
   \end{split}
\end{equation}
With this estimate we may finally apply the stochastic Gronwall lemma,
as in \cite{GlattHoltzZiane2} to conclude (\ref{eq:suffConForUniqueness}).
The proof is complete.
\end{proof}

\subsection{Compactness Revisited}
\label{sec:compRevist}

We return to the sequence $\{U^{n}\}$ of Galerkin solutions of
(\ref{eq:GalerkinCutoff}) defined relative to the given stochastic
basis $\mathcal{S}$.  We assume throughout this section that
$\E \|U_{0}\|^{q} < \infty$ for some $q \geq 8$.  Once we have 
established the existence of local pathwise solutions for all 
initial data in this class the general case, \eqref{eq:U0Cond1Pathwise}
may be established via a localization argument.  See e.g.
\cite{GlattHoltzZiane2}.

In pursuit of Proposition~\ref{thm:GyongyKry} we
consider the collection of joint distributions $\mu^{n,m}_U$ given by
$(U^{n}, U^{m})$.  For this purpose we define the extended phase space
(cf. \eqref{eq:phaseSpaceStrongConv})
\begin{equation}\label{eq:phaseSpaceStrongSolnConvExtend}
  \begin{split}
   \mathcal{X}^{J} 
      = \mathcal{X}_U \times \mathcal{X}_U \times \mathcal{X}_W,  
   &\quad 
   \mathcal{X}^{J}_U := \mathcal{X}_U \times \mathcal{X}_U, \\
   \mathcal{X}_U = L^2(0,T; V) \cap C([0,T], V'), 
   \quad \quad&
   \mathcal{X}_W = C([0,T],\mathfrak{U}_0).      
 \end{split}
\end{equation}
As above in \eqref{eq:solnMeasures}, \eqref{eq:BMMeasures} 
we let $\mu^n_U(E) = \Prb( U^n \in E)$ for $E \in \mathcal{X}_U$
and $\mu_W(E)  = \Prb( W \in E)$ for $E \in \mathcal{X}_W$.
Take
\begin{equation}\label{eq:solutionandStochLaws2}
 \begin{split}
 \mu^{n,m}_U = \mu^n_{U} \times \mu^m_{U}, \quad
 \nu^{n,m} = \mu^n_{U} \times \mu^m_{U} \times \mu_W.
\end{split}
\end{equation}
Similarly to Lemma~\ref{thm:tightness1pt0} we prove:
\begin{Lem}\label{thm:tightness2pt0}
  Suppose $\E \|U_{0}\|^{q} < \infty$ for some $q \geq 8$.
  The collection $\{\nu^{n,m}\}$ (and hence any subsequence
  $\{\nu^{n_k, m_k}\}$) is tight (and hence compact) on 
  $\mathcal{X}^{J}$.
\end{Lem}
\begin{proof}
  The proof is nearly identical to Lemma~\ref{thm:tightness1pt0}.  We
  determine the sets $B^1_R, B^2_R$ exactly as previously.  With trivial
  modifications (see (\ref{eq:ChebFinal1}) and remarks immediately
  above) we can therefore choose $A_\epsilon$, $\tilde{A}_\epsilon$
  compact in $\mathcal{X}_U$ and $\mathcal{X}_W$ respectively so that
  $\mu^n_U(A_\epsilon) \geq 1 - \frac{\epsilon}{4}$,  and
  $\mu^n_W(\tilde{A}_\epsilon) \geq 1 - \frac{\epsilon}{2}$,
  for every $n$.   Taking $\mathcal{K}_\epsilon := A_\epsilon \times
  A_\epsilon \times \tilde{A}_\epsilon$, which is compact in
  $\mathcal{X}^J$ we see that
  $\nu^{n,m}(\mathcal{K}_\epsilon) \geq 
    \left(1 -\frac{\epsilon}{4}\right)^2 \left(1 -\frac{\epsilon}{2}\right)
    \geq 1 - \epsilon$,
 which holds for every $0 < \epsilon< 1$ over all $m, n$.  The proof is complete.
\end{proof}

Suppose now that $\{\nu^{n_k, m_k}_U\}_{k \geq 0}$ is
any subsequence.  By Lemma~\ref{thm:tightness2pt0}, $\{\nu^{n_k,
  m_k}_U\}_{k \geq 0}$ is tight and hence by
\arxiv{Proposition~\ref{thm:ProhorovSkorhod} (i),}\physD{Prohorov's Theorem} we may choose as subsequence
$k'$ so that $\nu^{n_k',m_k'}$ converges weakly to an element $\nu'$.
By applying \arxiv{Proposition~\ref{thm:ProhorovSkorhod}, (ii)}\physD{Skorohod's Theorem} we next infer
the existence of a probability space $(\tilde{\Omega},
\tilde{\mathcal{F}}, \tilde{\Prb})$ upon which are defined a sequence
of random elements $(\tilde{U}^{n_k'},\tilde{\tilde{U}}^{m_k'}, \tilde{W}^{k'})$ converging almost
surely in $\mathcal{X}^J$ to an element $(\tilde{U},\tilde{\tilde{U}}, \tilde{W})$ 
in such a way that
$\tilde{\Prb}\left( (\tilde{U}^{n_k'},\tilde{\tilde{U}}^{m_k'}, \tilde{W}^{k'}) \in \cdot \right)
= \nu^{n_k', m_k'}(\cdot)$ and 
$\tilde{\Prb}\left( (\tilde{U},\tilde{\tilde{U}}, \tilde{W})  \in \cdot \right) = \nu^{'}(\cdot)$.
Let  $\tilde{Z}_{k'} = (\tilde{U}^{n_k'},
\tilde{W}^{k'})$, $\tilde{\tilde{Z}}_{k'} =
(\tilde{\tilde{U}}^{n_k'}, \tilde{W}^{k'})$, $\tilde{Z} = (\tilde{U},
\tilde{W})$ and $\tilde{Z}= (\tilde{\tilde{U}}, \tilde{W})$.
Note that in particular $\mu^{n_k,m_k}_U$ converges weakly to the measure $\mu_U$ defined by
\begin{equation}\label{eq:FinalJointPrbMeasure}
  \mu_U( \cdot ) :=  \tilde{\Prb}( (\tilde{U}, \tilde{\tilde{U}}) \in \cdot)
\end{equation}

Exactly as for Proposition~\ref{thm:SubsequentialConv},  we may establish the
conditions for Proposition~\ref{thm:PassageToLimitThm} below for both 
$\tilde{Z}_{k'}$, $\tilde{Z}$ and $\tilde{\tilde{Z}}_{k'}$, $\tilde{\tilde{Z}}$.  
As such we infer that both $\tilde{U}$ and
$\tilde{\tilde{U}}$ are global Martingale solutions of \eqref{eq:modFullSystemCut} over
the same stochastic basis $\tilde{S} = (\Omega, \mathcal{F},
\{\mathcal{F}_t \}_{t \geq 0}, \Prb, \tilde{W})$.  
Since it may be readily shown from the above convergences that 
$\tilde{U}(0) = \tilde{\tilde{U}}(0)$ a.s. we infer from Proposition~\ref{thm:Uniqueness} 
that $\tilde{U} = \tilde{\tilde{U}}$  in $\mathcal{X}_{U}$ a.s.  In other words
\begin{displaymath}
  \mu(\{ (x,y) \in \mathcal{X}^J_U \times \mathcal{X}^J_U : x = y\})
  = \tilde{\Prb}( \tilde{U} = \tilde{\tilde{U}}  \textrm{ in } \mathcal{X}_{U}) =1.
\end{displaymath}
With this conclusion, Proposition~\ref{thm:GyongyKry}, now implies
that the original sequence $U^n$ defined on the initial probability
space $(\Omega, \mathcal{F}, \Prb)$ converges to an element $U$, 
in the topology of $\mathcal{X}_{U}$.   By a final application of 
Proposition~\ref{thm:PassageToLimitThm}\footnote{Actually,
in contrast to previous cases above, the convergence is more 
straightforward since in this case we need to consider only one fixed
driving brownian motion $W$ throughout.} below we may infer that
$U$ is a global pathwise solution of \eqref{eq:modFullSystemCut}.
Hence, taking $\tau$ as in \eqref{eq:firstExistenceTime} below we 
conclude that $(U, \tau)$ is a local pathwise solution of \eqref{eq:AbsEq}.

The passage from a local to a maximal pathwise solution in the 
sense of Definition~\ref{def:Pathwisesolns}, (iii), may 
now be carried out as in \cite{GlattHoltzZiane1}.  See also \cite{Jacod1}.
This completes the proof of Theorem~\ref{thm:MainResult}.

\section{Example: The Primitive Equations of the Ocean}
\label{sec:examples}
\label{sec:prim-equat-ocean}

As discussed in the introduction the primary motivation in the
development of the abstract theory was to be able to treat 
the existence of local, pathwise solutions of the stochastic 
primitive equations of the ocean considered 
in the $\beta$-plane approximation.  We now recall this system
of equations and show how Theorem~\ref{thm:MainResult} 
applies to these equations.

The stochastic primitive equations of the ocean in their $\beta$-plane approximation
take the form
\begin{subequations}\label{eq:PE3DBasic}
  \begin{gather}
    \begin{split}
    \pd{t} \mathbf{v} 
    + (\mathbf{v} \cdot \nabla)\mathbf{v} + w \pd{z}\mathbf{v}
    + \frac{1}{\rho_0} \nabla p 
    + f \mathbf{k} \times \mathbf{v} 
    - \mu_{\mathbf{v}} \Delta \mathbf{v} 
    - \nu_{\mathbf{v}} \pd{zz} \mathbf{v}
    = F_{\mathbf{v}} + \sigma_{\mathbf{v}}(\mathbf{v},T,S) \dot{W}_1,
    \end{split}
    \label{eq:MomentumPE}\\
    \pd{z} p = - \rho g,
    \label{eq:HydroStaticPE}\\
    \nabla \cdot \mathbf{v} + \pd{z} w = 0,
    \label{eq:divFreeTypeCondPE}\\
    \pd{t} T + (\mathbf{v}\cdot \nabla) T
             + w \pd{z} T
             - \mu_{T} \Delta T
             - \nu_{T} \pd{zz} T
             = F_{T} + \sigma_{T}(\mathbf{v},T,S) \dot{W}_2,
    \label{eq:diffEqnTempPE}\\
    \pd{t} S + (\mathbf{v}\cdot \nabla) S
             + w \pd{z} S
             - \mu_{S} \Delta S
             - \nu_{S} \pd{zz} S
             = F_{S} + \sigma_{S}(\mathbf{v},T,S) \dot{W}_3,
    \label{eq:diffEqnSaltPE}\\
    \rho = \rho_0 ( 1 - \beta_T( T - T_r) + \beta_S(S- S_r)).
    \label{eq:linearDensityDependence}
  \end{gather}
\end{subequations}
Here, $\mathbf{v},T,S, p, \rho$ represent the horizontal
velocity,temperature, salt concentration, pressure and density of the 
fluid under consideration; $\mu_{\mathbf{v}}$, $\nu_{\mathbf{v}}$, $\mu_{T}$, 
$\nu_{T}$, $\mu_{S}$, $\nu_{S}$ are (possibly anisotropic) coefficients of the 
eddy viscosity and of the heat and salt diffusivity respectively; $f$ is the 
Coriolis parameter appearing
in the antisymmetric term in \eqref{eq:PE3DBasic} and accounts for the
earth's rotation in the momentum budget. The stochastic terms are
driven by white noise processes $\dot{W}_{j}$ and are understood in
the It\={o} sense.   The equations as given above model oceanic flows, however
equations of a quite similar structure may be given that describe the
atmosphere and the coupled oceanic atmospheric system, or the same
equations on the sphere.

Of course \eqref{eq:PE3DBasic} is supplemented
with appropriate boundary conditions which, among other considerations
must account for the coupling at the oceans surface with the
atmosphere.  The evolution occurs over a cylindrical domain
$\mathcal{M} =
\mathcal{M}_0 \times (-h, 0)$,  where $\mathcal{M}_0$ is an open
bounded subset of $\mathbb{R}^2$ with smooth boundary $\partial
\mathcal{M}_0$.  We denote by $\mathbf{n}_H$ the outward unit normal to
$\partial \mathcal{M}_0$  The boundary $\partial \mathcal{M}$ is partitioned
into the top $\Gamma_i = \mathcal{M}_0 \times \{0\}$, bottom
$\Gamma_b= \mathcal{M}_0 \times \{-h\}$  and sides $\Gamma_l = \partial
\mathcal{M}_0  \times (-h, 0)$. We prescribe (see \cite{LionsTemamWang2},
\cite{PetcuTemamZiane})
\begin{equation}\label{eq:3dPEPhysicalBoundaryCondTop}
  \begin{split}
    \nu_{\mathbf{v}}\pd{z} \mathbf{v} + \alpha_{\mathbf{v}} \mathbf{v}= 0, \quad w = 0 \quad
    \nu_{T}\pd{z} T + \alpha_{T} T = 0,\quad
    \pd{z} S = 0, \quad \textrm{ on } \Gamma_i,
  \end{split}
\end{equation}
\begin{equation}\label{eq:3dPEPhysicalBoundaryCondBottom}
  \begin{split}
  \mathbf{v} = 0, \quad w = 0, \quad
  \pd{z} T = 0, \quad \pd{z} S = 0, \quad \textrm{ on } \Gamma_b,    
  \end{split}
\end{equation}
\begin{equation}\label{eq:3dPEPhysicalBoundaryCondSide}
  \mathbf{v} = 0, \quad \pd{\mathbf{n}_H}T = 0, \quad \pd{\mathbf{n}_H}S = 0
  \quad \textrm{ on } \Gamma_l.
\end{equation}
The equations and boundary conditions \eqref{eq:PE3DBasic},
\eqref{eq:3dPEPhysicalBoundaryCondTop},
\eqref{eq:3dPEPhysicalBoundaryCondBottom},
\eqref{eq:3dPEPhysicalBoundaryCondSide} are supplemented 
by initial conditions for $\mathbf{v}$, $T$ 
and $S$, that is
\begin{equation}\label{eq:basicInitialCond}
   \mathbf{v} = \mathbf{v}_0, \quad
   T = T_0, \quad
   S = S_0, \quad
   \textrm{ at } t = 0.
\end{equation}
With the boundary conditions, 
\eqref{eq:3dPEPhysicalBoundaryCondTop}, 
\eqref{eq:3dPEPhysicalBoundaryCondBottom},
\eqref{eq:3dPEPhysicalBoundaryCondSide},
we may reformulate \eqref{eq:PE3DBasic} according to
\begin{subequations}\label{eq:PE3DInt}
  \begin{gather}
    \begin{split}
    \pd{t} \mathbf{v} 
    &+ (\mathbf{v} \cdot \nabla)\mathbf{v} + w(\mathbf{v}) \pd{z}\mathbf{v}
    + \frac{1}{\rho_0} \nabla p_s 
    - g \int_z^0 \left( \beta_T  \nabla T  + \beta_S  \nabla S  \right) d \bar{z}\\
    &+ f \mathbf{k} \times \mathbf{v} 
    - \mu_{\mathbf{v}} \Delta \mathbf{v} 
    - \nu_{\mathbf{v}} \pd{zz} \mathbf{v}
    = F_{\mathbf{v}} + \sigma_{\mathbf{v}}(\mathbf{v},T,S) \dot{W}_1,\\
    \end{split}
    \label{eq:MomentumPEint}\\
    w(\mathbf{v}) = \int^0_z \nabla \cdot \mathbf{v} d\bar{z},  \quad
    \int_{-h}^0 \nabla \cdot \mathbf{v} d\bar{z} = 0, 
    \label{eq:divFreeTypeCondPEInt}\\
    \pd{t} T + (\mathbf{v}\cdot \nabla) T
             + w(\mathbf{v}) \pd{z} T
             - \mu_{T} \Delta T
             - \nu_{T} \pd{zz} T
             = F_{T} + \sigma_{T}(\mathbf{v},T,S) \dot{W}_2,
    \label{eq:diffEqnTempPEInt}\\
    \pd{t} S + (\mathbf{v}\cdot \nabla) S
             + w(\mathbf{v}) \pd{z} S
             - \mu_{S} \Delta S
             - \nu_{S} \pd{zz} S
             = F_{S} + \sigma_{S}(\mathbf{v},T,S) \dot{W}_3,
    \label{eq:diffEqnSaltPEInt}
 \end{gather}
\end{subequations}
which is the basis for the functional framework that we next recall.
Our presentation and notations closely follow the recent survey
\cite{PetcuTemamZiane} which covers the deterministic setting.

We will denote by $U$ the triple of prognostic variables, $U = (\mathbf{v}, T,S)$
(comprising four scalar variables), and we set
\begin{displaymath}
  \begin{split}
    H = \biggl\{ (\mathbf{v}, T,S) \in L^{2}(\mathcal{M})^{4} :
	 \nabla \cdot \int_{-h}^{0} \mathbf{v} dz = 0 \textrm{ over } \mathcal{M}_{0},
	 \mathbf{n}_H \cdot  \int_{-h}^{0} \mathbf{v} dz = 0 \textrm{ over } \partial \mathcal{M}_{0}, 
	 \int_{\mathcal{M}} S d\mathcal{M} = 0 \biggr\},
   \end{split}
\end{displaymath}
which we equip with the classical $L^{2}$ inner product\footnote{One sometimes also finds the more general definition
  $(U, U^{\sharp}) := \int_{\mathcal{M}} (\mathbf{v}\cdot
  \mathbf{v}^\sharp d + \kappa_{T}  T T^\sharp   + \kappa_{S} S S^{\sharp} )d
  \mathcal{M}$ with $\kappa_{T}, \kappa_{S} > 0$ fixed constants. These parameters $\kappa_{T}, \kappa_{S}$
  are useful for the coherence of physical dimensions and for (mathematical)
  coercivity. Since this is not needed here we take $\kappa_{T} = \kappa_{S} =1$.
  Similar remarks also apply to the space $V$.}.
Define $P_{H}$ to be the Leray type projection operator from $L^{2}(\mathcal{M})^{4}$ onto $H$.   For $H^{1}(\mathcal{M})^{4}$
we consider the subspace:
\begin{displaymath}
  \begin{split}
    V = \biggl\{ (\mathbf{v}, T,S) \in H^{1}(\mathcal{M})^{4} :
	 \nabla \cdot \int_{-h}^{0} \mathbf{v} dz = 0 \textrm{ over } \mathcal{M}_{0},
	 \mathbf{v} = 0  \textrm{ on } \Gamma_{l} \cup \Gamma_{b},
	 \int_{\mathcal{M}} S d\mathcal{M} = 0 \biggr\}.
  \end{split}
\end{displaymath}
We equip $V$ with the inner product 
\begin{displaymath}
  \begin{split}
    ((U, U^{\sharp})) :=& ((\mathbf{v},\mathbf{v}^{\sharp}))_{1} 
          +((T,T^{\sharp}))_{2} 
          +((S,S^{\sharp}))_{3},\\
    ((\mathbf{v},\mathbf{v}))_{1} :=&
    \int_{\mathcal{M}} \left( \mu_{\mathbf{v}}
     \nabla\mathbf{v} \cdot \nabla\mathbf{v}^{\sharp}+ 
     \nu_{\mathbf{v}}
     \pd{z}\mathbf{v} \cdot \pd{z}\mathbf{v}^{\sharp}
     \right) d \mathcal{M}
     +  \alpha_{\mathbf{v}} \int_{\Gamma_i}
     \mathbf{v}\mathbf{v}^{\sharp} d \Gamma_i,\\
     ((T,T^{\sharp}))_{2} :=&
   \int_{\mathcal{M}} \left(\mu_{T}
    \nabla T \cdot \nabla T^{\sharp}+ \nu_{T}
    \pd{z} T \cdot \pd{z} T^{\sharp}
    \right) d \mathcal{M}
    +  \alpha_{T} \int_{\Gamma_i}
     T T d \Gamma_i,\\
     ((S,S^{\sharp}))_{3} :=&
   \int_{\mathcal{M}} \left(\mu_{S}
    \nabla S \cdot \nabla S^{\sharp}+ \nu_{S}
    \pd{z} S \cdot \pd{z} S^{\sharp}
    \right) d \mathcal{M}.
\end{split}
\end{displaymath}
Note that under these definitions a Poincar\'{e} type 
inequality  $|U|\leq C \|U\|$ holds for all $U \in V$.  
We take $V_{(2)}$ to be the closure of $V$ in the $H^2(\mathcal{M})^4$ norm
and
equip this space with the classical $H^2(\mathcal{M})$ norm and inner
product.

The main linear portion of the equation is defined by\footnote{In \cite{PetcuTemamZiane}
a slightly different non-selfadjoint operator also called $A$ is given which corresponds to
$A + A_{p}$ (cf. \eqref{eq:linLowerOrderDef}) in the present  manuscript.}
\begin{displaymath}
  AU = P_H \left(
  \begin{split}
    -\mu_{\mathbf{v}} \Delta \mathbf{v} - \nu_{\mathbf{v}} \pd{zz} \mathbf{v} \\
    -\mu_{T} \Delta T - \nu_{T} \pd{zz} T \\
    -\mu_{S} \Delta S - \nu_{S} \pd{zz} S \\
 \end{split}
  \right), \quad U = (\mathbf{v}, T, S) \in D(A).
\end{displaymath}
where we take:
\begin{displaymath}
  \begin{split}
    D(A) = \{ U = (\mathbf{v}, T) \in V_{(2)}: \;&  
       \nu_{\mathbf{v}} \pd{z} \mathbf{v} +\alpha_{\mathbf{v}} \mathbf{v} = 0,
       \nu_{T} \pd{z} T + \alpha_T T = 0, 
       \pd{z} S = 0 \textrm{ on } \Gamma_i,\\
    &\pd{\mathbf{n}_H} T =  \pd{\mathbf{n}_H} S = 0 \textrm{ on } \Gamma_l,  
       \pd{z} T = \pd{z} S = 0 \textrm{ on } \Gamma_b  
       \}.
  \end{split}
\end{displaymath}
We observe that $A$ satisfies the conditions given in
Section~\ref{sec:absSpOps}.  Note also that due to \cite{Ziane1} 
(see also \cite{PetcuTemamZiane}) $|AU| \cong  |U|_{H^2}$.

We next turn to the quadratically nonlinear terms appearing in \eqref{eq:PE3DInt}.
Noting that there is no momentum equation for $w$ in
\eqref{eq:PE3DInt} and in accordance with \eqref{eq:divFreeTypeCondPEInt}
we \emph{define} the diagnostic function:
\begin{displaymath}
  w(U) = w(\mathbf{v}) = \int^0_z \nabla \cdot \mathbf{v} d \bar{z}, \quad
  U = (\mathbf{v}, T, S) \in V.
\end{displaymath}
Take, for $U, U^{\sharp} \in D(A)$:
\begin{equation}\label{eq:NLTerm1}
  B_1(U,U^\sharp) :=  P_H \left( 
  \begin{split}
    (\mathbf{v} \cdot \nabla)\mathbf{v}^\sharp\\
   (\mathbf{v} \cdot \nabla) T^\sharp\\
   (\mathbf{v} \cdot \nabla) S^\sharp\\
  \end{split}
  \right), \quad
  B_2(U,U^\sharp) :=  P_H \left( 
  \begin{split}
    w(\mathbf{v}) \pd{z} \mathbf{v}^\sharp\\
    w(\mathbf{v}) \pd{z} T^\sharp \\
    w(\mathbf{v}) \pd{z} S^\sharp \\
  \end{split}
  \right)
 \end{equation}
and let $B(U, U^\sharp) := B(U, U^\sharp) + B(U, U^\sharp)$.   
As in \cite{PetcuTemamZiane} one may show that $B$
is well defined as an element in $H$ for any $U, U^{\sharp} \in D(A)$.
Furthermore $B$ satisfies the conditions
\eqref{eq:Bprop1}, \eqref{eq:Bprop2},  \eqref{eq:Bprop3} relative to
the definitions of $A$, $D(A)$, $V$ and $H$ given here. For the second
component of the pressure in \eqref{eq:MomentumPEint} we take
\begin{equation}\label{eq:linLowerOrderDef}
  A_p U =   P_H \left(
  \begin{array}{c}
    - g \int_z^0 \left( \beta_T  \nabla T  + \beta_S  \nabla S  \right) d \bar{z} \\
    0\\
    0
  \end{array}
  \right), \quad U \in V.
\end{equation}
and capture the Coriolis forcing in
\begin{equation}
  \label{eq:CorTerm}
  E U = P_H \left(
  \begin{array}{c}
    f \mathbf{k} \times v \\
    0\\
    0
  \end{array}
  \right), \quad U \in H.
\end{equation}
Finally we set
\begin{equation}\label{eq:DetForcingTerm}
    F_{U} = P_H \left(
  \begin{array}{c}
    F_{\mathbf{v}}\\
    F_{T}\\
    F_{S}\\
  \end{array}
  \right).
\end{equation}
We may therefore define
\begin{equation}\label{eq:lowerorderLin}
  F(U) = A_p U + E U + F_{U}
\end{equation}
and observe that $F: V \rightarrow H$ and satisfies 
the requirement \eqref{eq:lipConF}.  Finally we define
\begin{equation}\label{eq:nonLinearNoisePE}
    \sigma((\mathbf{v}, T,S)) = \sigma(U) = P_{H}
    \left(
      \begin{array}{c}
    \sigma_{\mathbf{v}}(U)\\
    \sigma_{T}(U)\\
    \sigma_{S}(U)\\
  \end{array}
      \right),  \quad U \in H,
\end{equation}
and assume either \eqref{eq:BndCondSig}  or
\eqref{eq:lipCondSig} for the consideration of
Martingale or Pathwise solutions receptively.  

With the above definitions in place we may write \eqref{eq:PE3DBasic}
supplemented by the boundary conditions 
\eqref{eq:3dPEPhysicalBoundaryCondTop},
\eqref{eq:3dPEPhysicalBoundaryCondBottom},
\eqref{eq:3dPEPhysicalBoundaryCondSide}
in the abstract form \eqref{eq:AbsEq} and conclude via 
Theorem~\ref{thm:MainResult}:
\begin{Thm}\label{thm:MainResPE}
    Assume that $F_{\mathbf{v}}$, $F_{T}$ and
    $F_{S}$ are in $L^{2}_{loc}([0,\infty), L^{2}(\mathcal{M}))$
    and suppose that $\sigma(\cdot)$ associated to 
    $\sigma_{\mathbf{v}}(\cdot)$, $\sigma_{T}(\cdot)$,
    $\sigma_{S}(\cdot)$ via  \eqref{eq:nonLinearNoisePE}
    satisfies \eqref{eq:BndCondSig}.  Finally 
    suppose that $(\mathbf{v}_{0}, T_{0}, S_{0})$ takes values 
    in $V$ and that
    $\mu_{0}(\cdot) = \mathbb{P}( (\mathbf{v}, T, S) \in \cdot )$ 
    satisfies the moment condition \eqref{eq:U0CondMartingale}
    with $q \geq 8$.  Then:
    \begin{itemize}
    	\item[(i)] There exists a local martingale solution
	of \eqref{eq:PE3DBasic}, the primitive equations of
	the ocean.
	\item[(ii)] If we additionally assume \eqref{eq:lipCondSig}
	for $\sigma$ and allow of the relaxation of 
	\eqref{eq:U0CondMartingale} to cover any $q \geq 2$ then
	there exists a unique maximal, pathwise solution of 
	\eqref{eq:PE3DBasic}.
    \end{itemize}
\end{Thm} 

\begin{Rmk}
As noted in the introduction the Stochastic Navier-Stokes equations
have been extensively studied.  Initially these equations were considered
with an additive noise.  
See \cite{BensoussanTemam} 
and later,  \cite{ZabczykDaPrato2, MenaldiSritharan}.  
In this case, a classical transformation,
allows one to treat $\omega \in \Omega$ as a parameter in the problem.  
For nonlinear multiplicative noise, the Navier-Stokes equations were 
initially studied in the context of Martingale
solutions.  See,  for example, \cite{Viot1,Cruzeiro1,
CapinskiGatarek, Flandoli1, MikuleviciusRozovskii4}.  
These works typically considered, from the PDE point of view, weak 
solutions evolving in time only in $L^2_x$.
More recently, pathwise solutions (both local and global in time) for a multiplicative noise have
been investigated.  See \cite{Breckner, BensoussanFrehse, BrzezniakPeszat,
MikuleviciusRozovskii2}.  None of these works addressed
the existence and uniqueness of a local pathwise solution in $3D$ evolving in $H^{1}_x$. 
This case, whose deterministic analogue would be considered classical,
was established only recently in \cite{GlattHoltzZiane2}.
As in the present work a key difficulty involves finding suitable compactness 
methods to pass to the limit.  In \cite{GlattHoltzZiane2}  the necessary compactness
is established by directly showing that the sequence of Galerkin solutions
are Cauchy.   For an application of this approach
to the 2D Primitive equations see \cite{GlattHoltzTemam1, GlattHoltzTemam2}.
In any case the work here provides an alternative proof of the results in
\cite{GlattHoltzZiane2}.  The reader may readily check that the abstract framework 
developed in Section~\ref{sec:math-fram} applies to the Navier-Stokes equations
on a bounded domain with Dirichlet boundary conditions and with $H$ and $V$
corresponding, approximately speaking, to $L^{2}_x$ and $H^{1}_x$.
\end{Rmk}

\section{The Passage to the Limit}
\label{sec:passToLim}

In this section we provide the details of the passage to the limit,
which is used in the proof of the existence of both
martingale solutions and pathwise solutions.  See
Proposition~\ref{thm:SubsequentialConv} and 
Section~\ref{sec:compRevist} above.

\begin{Prop}\label{thm:PassageToLimitThm}
  Let $Z_k = (\tilde{U}^{m_k},  \tilde{W}^{m_k})$ be a sequence of $\mathcal{X}$
  valued random elements mapping from a probability space  $(\tilde{\Omega}, \tilde{\mathcal{F}},
  \tilde{\Prb})$.  We assume that
  \begin{itemize}
  \item[(i)] $Z_k$ converges almost surely to an element $Z$ in the
    topology of $\mathcal{X}$, i.e.
    \begin{subequations}
      \begin{gather}
        \tilde{U}^{m_k} \rightarrow \tilde{U}  \quad \textrm{ in }
        L^2([0,T], V) \cap C([0,T]; V'),
        \label{eq:meaningoftopX}\\
        \tilde{W}^{m_k} \rightarrow \tilde{W} \quad \textrm{ in }
        C([0,T]; \mathfrak{U}_0).
        \label{eq:meaningoftopW}
      \end{gather}
    \end{subequations}
  \item[(ii)] Each $\tilde{W}^{m_k}$ is a cylindrical Wiener process
    relative to a filtration $\mathcal{F}_t^{m_k}$ that contains 
    $\sigma((\tilde{W}^{m_k}(s), \tilde{U}^{m_k}(s));$ $s \leq t)$.
 \item[(iii)] Each pair $(\tilde{U}^{n_k}, \tilde{W}^{n_k})$ satisfies
    \begin{equation}\label{eq:GalerkinCutoffGivenbasis}
      \begin{split}
        d \tilde{U}^{n_k} + [A \tilde{U}^{n_k} + \theta( \| \tilde{U}^{n_k} - \tilde{U}_*^{n_k} \| ) 
           &B^{n_k}(\tilde{U}^{n_k}) + F^{n_k}(\tilde{U}^{n_k}) ]dt
           =\sigma^{n_k}(\tilde{U}^{n_k}) d\tilde{W}^{n_k},\\
           \tilde{U}^{n_k}(0) &= P_{n_k} \tilde{U}_0 := \tilde{U}^{n_k}_0
      \end{split}
    \end{equation}
    where we define $ \tilde{U}_*^{n_k}$ by
    \begin{equation}\label{eq:AuxLinearmodGivenBasis}
      \frac{d}{dt} \tilde{U}_*^{n_k} + A \tilde{U}_*^{n_k} = 0 
      \quad \tilde{U}_*^{n_k}(0) = \tilde{U}^{n_k}_0,
    \end{equation}
    and assume, for some $p >4$, that
    \begin{equation}\label{eq:dataCondForConv}
        \E \|\tilde{U}_{0}\|^{p} < \infty.
    \end{equation}
 \end{itemize}
 Let $\tilde{\mathcal{S}} = (\tilde{\Omega}, \tilde{\mathcal{F}},\{\tilde{\mathcal{F}}_t\}_{t\geq 0},
  \tilde{\Prb}, \tilde{W})$, defining $\tilde{\mathcal{F}}_t$ as the
 completion of  $\sigma(\tilde{W}(s), \tilde{U}(s); s \leq
    t)$.  Then $(\tilde{\mathcal{S}},\tilde{U})$ is a global martingale solution of \eqref{eq:modFullSystemCut}.
    Moreover if we define the stopping time:
    \begin{equation}\label{eq:firstExistenceTime}
      \tau := \inf_{t \geq 0} \left\{ \|\tilde{U} - \tilde{U}_* \| \geq \kappa \right\},
    \end{equation}
    where $\kappa$ is constant appearing in the definition of $\theta$, \eqref{eq:CuttoffDef},
    then $(\tilde{\mathcal{S}}, \tilde{U}, \tilde{\tau})$ is a local martingale solution of (\ref{eq:AbsEq}).
\end{Prop}

The rest of this section is devoted to the proof which proceeds in stages.  
The first step is to establish that the candidate solution is in better spaces 
via \eqref{eq:UniformGalEst1} and weak compactness arguments.
We next establish the almost sure limits of each terms arising in
\eqref{eq:GalerkinCutoffGivenbasis} against sufficiently `smooth' test
functions $U^\sharp \in D(A)$. We then show,
using (\ref{eq:UniformGalEst1}) for $\tilde{U}^{n_k}$
and the Vitali convergence theorem (see e.g. \cite{Folland1}) that each of the deterministic terms
converges in $L^2(\Omega \times [0,T])$.  The convergence of the
stochastic terms in \eqref{eq:GalerkinCutoffGivenbasis} are
facilitated by Lemma~\ref{thm:ConvThm}. With these convergences in hand
we make use of a variational argument (se e.g. \cite{GlattHoltzZiane2}) to finally 
conclude \eqref{eq:modFullSystemCut}
for almost every time $t$ and $\omega \in \tilde{\Omega}$ and with the equality understood in $H$.  
We pass to the limit for every $t$ by establishing the improved continuity of $U$.  See 
Subsection~\ref{sec:impr-regul-time} below.  This improved continuity 
justifies the definition of the stopping time $\tau$ specified by 
\eqref{eq:firstExistenceTime}.
We therefore infer that for every $t \geq 0$,
\begin{equation}\label{eq:LookMaNoCutoffLocally}
   \int_0^{t \wedge \tau} \theta(\| \tilde{U} - \tilde{U}_* \| ) B( \tilde{U}) ds
   =
   \int_0^{t \wedge \tau} B( \tilde{U}) ds.
\end{equation}
In this manner we finally conclude that $(\tilde{S}, \tilde{U}, \tau)$
is a local martingale solution of \eqref{eq:AbsEq} and complete the proof.

\subsection{Improved Regularity of the Candidate Solution}
\label{sec:RegularityCand}

By applying the Banach - Alaoglu theorem with
(\ref{eq:UniformGalEst1}) for $\tilde{U}^{n_k}$ in the case $p = 2$ we
infer the existence of elements $\hat{\tilde{U}} \in L^2(\tilde{\Omega};
L^2([0,T], D(A))$ and $\hat{\hat{\tilde{U}}} \in L^2(\tilde{\Omega};
L^\infty([0,T],V))$ such that
\begin{equation}\label{eq:weakConvDA1}
  \tilde{U}^{n_k} \rightharpoonup \hat{\tilde{U}}
  \quad \textrm{ in } 
  L^2(\tilde{\Omega};L^2([0,T], D(A)),
\end{equation}
and
\begin{equation}\label{eq:weakStarConvV1}
  \tilde{U}^{n_k} \rightharpoonup^* \hat{\hat{\tilde{U}}}
  \quad \textrm{ in } 
  L^2(\tilde{\Omega};L^\infty([0,T], V).
\end{equation}
On the other hand, due to \eqref{eq:dataCondForConv},
applied to Lemma~\ref{thm:UniformEstimateGalerkin},
\eqref{eq:UniformGalEst1} we infer that
for some $q > 2$,
\begin{equation}  \label{eq:dominationofVprm}
  \sup_{k} \E \sup_{t \in [0,T]}|\tilde{U}^{n_k}|_{V'}^q
  \leq c \sup_{k} \E \sup_{t \in [0,T]}\|\tilde{U}^{n_k}\|^q < \infty.
\end{equation}
Thus, with \eqref{eq:meaningoftopX}, the Vitali convergence theorem implies that,
\begin{equation}\label{eq:finalConvVall}
  \tilde{U}^{n_k} \rightarrow \tilde{U}
  \quad
  \textrm{in } L^2(\tilde{\Omega}, L^\infty(0,T;V')).
\end{equation}
Take $\mathcal{R} \subset [0,T] \times \Omega$, measurable and
$U^\sharp \in D(A)$.  By applying \eqref{eq:weakConvDA1},
\eqref{eq:weakStarConvV1}, \eqref{eq:finalConvVall}, we find
\begin{equation}\label{eq:finalVarEqualityCandidates}
  \E\int_0^T \chi_{R}\langle \tilde{U}, U^\sharp \rangle ds
  =  \E\int_0^T \chi_{R}\langle\hat{\hat{\tilde{U}}}, U^\sharp \rangle ds
  =  \E\int_0^T \chi_{R}\langle \hat{\tilde{U}}, U^\sharp\rangle ds,
\end{equation}
which means that $\tilde{U} = \hat{\hat{\tilde{U}}} = \hat{\tilde{U}}$ and we conclude that
\begin{equation}\label{eq:RegConclusionsPart1}
  \tilde{U} \in L^2(\tilde{\Omega}, L^2([0,T], D(A))) 
                \cap L^2(\tilde{\Omega}, L^\infty([0,T], V)).
\end{equation}
Furthermore with \eqref{eq:weakConvDA1} we have
\begin{equation}\label{eq:weakConvDA1ActualLim}
  \tilde{U}^{n_k} \rightharpoonup \tilde{U}
  \quad \textrm{ in } 
  L^2(\tilde{\Omega};L^2([0,T], D(A)).
\end{equation}

\subsection{Variational Equality for the Cutoff System}
\label{sec:VerEQ}

Fix $U^\sharp \in D(A)$.  Since, almost surely,
$\tilde{U}^{n_k} \rightarrow \tilde{U}$ in $L^2([0,T],V)$ and noting
that
\begin{displaymath}
  \begin{split}
    \sup_k\E \left[
    	\left( \int_0^T \| \tilde{U}^{n_k} \|^2 dt \right)^2
	\right]
     \leq& \sup_k c \E \left( \sup_{t \in [0,T]} \| \tilde{U}^{n_k}
       \|^4 \right) < \infty
  \end{split}
\end{displaymath}
we infer that $\tilde{U}^{n_k} \rightarrow \tilde{U}$ in $L^2(\Omega,
L^2([0,T],V))$, by the Vitali convergence theorem. By thinning
the sequence $n_k$ if necessary, we may also conclude that
\begin{equation}\label{eq:pointwiseTimeConv}
  \|\tilde{U}^{n_k} - \tilde{U}\|^2 \rightarrow 0,
\end{equation}
for almost every $(t, \omega) \in [0,T] \times \tilde{\Omega}$.

The pointwise convergence in the linear term is direct:
\begin{displaymath}
  \left|
    \int_0^t \langle A (\tilde{U}^{n_k} - \tilde{U}), U^\sharp \rangle ds
    \right|
    \leq c \|U^\sharp\| \left(
    \int_0^T \|\tilde{U}^{n_k} - \tilde{U}\|^2 ds \right)^{1/2}.
\end{displaymath}
We conclude that for almost every $(t, \omega) \in [0,T] \times \tilde{\Omega}$
\begin{equation}\label{eq:FinalAConv}
  \int_0^t \langle A \tilde{U}^{n_k}, U^\sharp \rangle ds
  \rightarrow 
  \int_0^t \langle A \tilde{U}, U^\sharp \rangle ds.
\end{equation}
For $B$ we estimate
\begin{displaymath}
  \begin{split}
    \biggl|
      \int_0^t \langle
        &\theta(\|\tilde{U}^{n_k} - \tilde{U}^{n_k}_* \|) B^{n_k}(\tilde{U}^{n_k})-
        \theta(\|\tilde{U} - \tilde{U}_* \|) B(\tilde{U}), U^\sharp
      \rangle ds \biggr|\\
      \leq&
          \left|
      \int_0^t \langle
        \theta(\|\tilde{U}^{n_k} - \tilde{U}^{n_k}_* \|) (B^{n_k}(\tilde{U}^{n_k})-
        B(\tilde{U})), U^\sharp
      \rangle ds \right|
         +
          \left|
      \int_0^t \langle
        (\theta(\|\tilde{U}^{n_k} - \tilde{U}^{n_k}_* \|)-
        \theta(\|\tilde{U} - \tilde{U}_* \|))
        B(\tilde{U}), U^\sharp
      \rangle ds \right|\\
      \leq&
          \left|
      \int_0^t \langle
        \theta(\|\tilde{U}^{n_k} - \tilde{U}^{n_k}_* \|) (B(\tilde{U}^{n_k})-
        B(\tilde{U})), P_{n_k} U^\sharp
      \rangle ds \right|
         +
         \left|
           \int_0^t \langle
        \theta(\|\tilde{U}^{n_k} - \tilde{U}^{n_k}_* \|)
        B(\tilde{U}), Q_{n_k} U^\sharp
      \rangle ds \right|\\
         &+
          \left|
            \int_0^t \langle
        (\theta(\|\tilde{U}^{n_k} - \tilde{U}^{n_k}_* \|)-
        \theta(\|\tilde{U} - \tilde{U}_* \|))
        B(\tilde{U}), U^\sharp
      \rangle ds \right|\\
    :=& J_1^{n_k} + J_2^{n_k} + J_3^{n_k}.
  \end{split}
\end{displaymath}
We address the elements on the right hand side in reverse order.  
Due to \eqref{eq:pointwiseTimeConv} we have, for almost
every $(t, \omega) \in [0,T] \times \tilde{\Omega}$:\footnote{Since 
   $\tilde{U}_0^{n} \rightarrow \tilde{U}_0$ in $L^2(\tilde{\Omega};V)$, 
   it follows that $\E \left( \sup_{t\in [0,T]} \|\tilde{U}^{n}_{*} -
      \tilde{U}_{*}\|^2 \right) \rightarrow 0.$}
\begin{displaymath}
  \| \tilde{U}^n- \tilde{U}^n_*\| 
  \rightarrow \| \tilde{U} - \tilde{U}_*\|.
\end{displaymath}
We therefore infer,
\begin{displaymath}
  \theta(\| \tilde{U}^n - \tilde{U}^n_*\|)
  \rightarrow \theta(\| \tilde{U} - \tilde{U}_*\|),
\end{displaymath}
almost everywhere on $[0,T] \times  \tilde{\Omega}$.
By assumptions (\ref{eq:Bprop1}),
(\ref{eq:Bprop2}),
\begin{displaymath}
  \left| \langle
      (\theta(\|\tilde{U}^{n_k} - \tilde{U}^{n_k}_* \|)-
      \theta(\|\tilde{U} - \tilde{U}_* \|))
      B(\tilde{U}), U^\sharp \rangle \right|dt
  \leq c |A U^\sharp| \|\tilde{U}\|^2,
\end{displaymath}
and since,
\begin{displaymath}
  \begin{split}
  \E \int_0^T \int_0^t &\left| \langle
       (\theta(\|\tilde{U}^{n_k} - \tilde{U}^{n_k}_* \|)-
       \theta(\|\tilde{U} - \tilde{U}_* \|))
       B(\tilde{U}), U^\sharp \rangle \right|ds dt\\
     &\leq c   \E \int_0^T  \left| \langle
      (\theta(\|\tilde{U}^{n_k} - \tilde{U}^{n_k}_* \|)-
      \theta(\|\tilde{U} - \tilde{U}_* \|))
      B(\tilde{U}), U^\sharp \rangle \right|dt,\\
\end{split}
\end{displaymath}
the Lebesgue dominated convergence theorem therefore implies:
\begin{displaymath}
  \E \int_0^T J_3^{n_k}dt \rightarrow 0.
\end{displaymath}
Thinning the sequence, if necessary we conclude that
\begin{equation}\label{eq:J3conclusionVERyIMPortant1}
   J_3^{n_k} \rightarrow 0 \quad a.e. \;\; (t, \omega) 
   \in [0,T] \times \tilde{\Omega}.
\end{equation}
We estimate the second term, $J_2^{n_k}$ according to
\begin{equation}\label{eq:BestJ2obs1}
  J_2^{n_k} \leq 
  | Q_{n_k} A U^\sharp | \int_0^T \|\tilde{U}\|^2 ds
  \rightarrow 0,
\end{equation}
for almost every $(t, \omega)$.  Finally for $J^{n_{k}}_{1}$, the bilinearity of $B$ implies that
\begin{displaymath}
  \begin{split}
  B(\tilde{U}^{n_k}, \tilde{U}^{n_k}) - B(\tilde{U}, \tilde{U})
    = B(\tilde{U}^{n_k}- \tilde{U}, \tilde{U}^{n_k}) 
      + B(\tilde{U},\tilde{U}^{n_k} - \tilde{U}).
  \end{split}
\end{displaymath}
Again by the assumptions (\ref{eq:Bprop1}), (\ref{eq:Bprop2}) we infer,
\begin{displaymath}
  \begin{split}
    J_1^{n_k}
    \leq& c
    |A U^\sharp| \int_0^T ( \|\tilde{U}^{n_k}\|  + \|\tilde{U}\|)
    \| \tilde{U}^{n_k} -\tilde{U}\| ds\\
    \leq& c
    |A U^\sharp| 
    \left(\int_0^T ( \|\tilde{U}^{n_k}\|^{2}  + \|\tilde{U}\|^{2}) \, ds \right)^{1/2}
    \left(\int_0^T \| \tilde{U}^{n_k} 
          -\tilde{U}\|^2 ds \right)^{1/2}.\\
  \end{split}
\end{displaymath}
Thus, with assumption \eqref{eq:meaningoftopX}, we infer that
\begin{equation}\label{eq:J1conclusionVERyIMPortant}
 J_1^{n_k} \rightarrow 0 \quad 
 \textrm{ for almost all } (t, \omega).
\end{equation}
Combining \eqref{eq:J1conclusionVERyIMPortant},
\eqref{eq:BestJ2obs1}, \eqref{eq:J3conclusionVERyIMPortant1} we
conclude that, for almost every $(t, \omega)$,
\begin{equation}\label{eq:FinalBconvConc}
  \int_0^t \langle
  \theta( \|\tilde{U}^{n_k} - \tilde{U}^{n_k}_* \|) B^n(\tilde{U}^n),
  U^\sharp \rangle ds
  \rightarrow
  \int_0^t \langle
  \theta( \|\tilde{U} - \tilde{U}_* \|) B(\tilde{U}),
  U^\sharp \rangle ds.
\end{equation}

For the remaining deterministic terms we estimate
\begin{equation}\label{eq:FlowerOrder1}
  \begin{split}
    \biggl| \int_0^t \langle F^{n_k}(\tilde{U}^{n_k}) - F(\tilde{U}), U^\sharp \rangle
      ds
      \biggr|
      \leq& c
       |U^\sharp|\int_0^t |F(\tilde{U}^{n_k}) - F(\tilde{U})|ds
       + |Q_{n_k} U^\sharp| \int_0^t |F(\tilde{U})|ds 
       := J_4^{n_k} + J_5^{n_k}.
  \end{split}
\end{equation}
Due to (\ref{eq:pointwiseTimeConv}) and the continuity assumed
for $F$, \eqref{eq:sublinerConF}, we infer that for almost every
$(\omega, t)$,
\begin{displaymath}
  |F(\tilde{U}^{n_k}) - F(\tilde{U})| \rightarrow 0. 
\end{displaymath}
On the other hand by (\ref{eq:sublinerConF}),
\begin{displaymath}
  |F(\tilde{U}^{n_k}) - F(\tilde{U})| 
    \leq c (1 + \|\tilde{U}^{n_k}\| + \|\tilde{U}\|),
\end{displaymath}
and we infer that,
\begin{displaymath}
  \begin{split}
 \sup_k \E \int_0^T &
    |F(\tilde{U}^{n_k}) - F(\tilde{U})|^{2} dt
  \leq c \sup_k
   \E \int_0^T(1 + \|\tilde{U}^{n_k}\|^2 + \|\tilde{U}\|^2) dt < \infty.\\
  \end{split}
\end{displaymath}
In consequence $\{ |F(\tilde{U}^{n_k}) - F(\tilde{U})| \}_{k \geq0}$
is uniformly integrable over $\tilde{\Omega} \times [0,T]$.  By
applying the Vitali Convergence Theorem we have 
$\E \int_{0}^{T} | F^{n_{k}}(\tilde{U}^{n_{k}}) - F(\tilde{U}) | dt \rightarrow 0$.
Thinning the sequence further if needed, we infer that almost everywhere
in $\tilde{\Omega}$.
\begin{displaymath}
  \begin{split}
	\int_{0}^{t} | F^{n_{k}}(\tilde{U}^{n_{k}}) - F(\tilde{U}) | dt 
	\leq	\int_{0}^{T} | F^{n_{k}}(\tilde{U}^{n_{k}}) - F(\tilde{U}) | dt \rightarrow 0.
  \end{split}
\end{displaymath}
in order to finally conclude that
\begin{equation}  \label{eq:FlowerOrderJ4Concl}
  J_4^{n_k} \rightarrow 0 \quad a.e. \; (\omega, t).
\end{equation}
Turning to the second term $J_5^{n_k}$, we see, again as a consequence
of the assumption  (\ref{eq:sublinerConF}), 
that $\int_0^t
|F(\tilde{U})|ds \leq c \int_0^T (1 + \|\tilde{U}\|)ds < \infty$ and so
\begin{equation}  \label{eq:FlowerOrderJ5Concl}
  J_5^{n_k} \rightarrow 0 \quad
   a.e. \, (\omega, t).
\end{equation}
In conclusion, by  (\ref{eq:FlowerOrderJ4Concl}),
(\ref{eq:FlowerOrderJ5Concl}) we finally have
\begin{equation}\label{eq:finalfinalfinalFpntWise}
  \int_0^t \langle F^{n_k}(\tilde{U}^{n_k}), U^\sharp \rangle ds
  \rightarrow \int_0^t \langle F(\tilde{U}), U^\sharp \rangle ds
\end{equation}
for almost every $(\omega, t) \in \tilde{\Omega} \times [0,T]$.

We next establish the convergences to the deterministic terms in
\eqref{eq:modFullSystemCut} in the space
$L^q(\tilde{\Omega} \times [0,T])$, $1 \leq q < 2$.  Notice that
due to \eqref{eq:Bprop2}, \eqref{eq:sublinerConF},
\begin{displaymath}
  \begin{split}
  \E \int_0^T & 
     \left|
       \int_0^t 
       \langle
       A\tilde{U}^{n_k}+
       \theta( \|\tilde{U}^{n_k}- \tilde{U}^{n_k}_* \|)
         B^n(\tilde{U}^{n_k})+
       F^n(\tilde{U}^{n_k}), U^\sharp
       \rangle ds
       \right|^2dt\\
       \leq& c
   \E \int_0^T 
     |\langle
       A\tilde{U}^{n_k}+
       \theta( \|\tilde{U}^{n_k}- \tilde{U}^{n_k}_* \|)
         B^n(\tilde{U}^{n_k})+
       F^n(\tilde{U}^{n_k}), U^\sharp
       \rangle|^2 ds\\
       \leq& c | AU^\sharp |^2
   \E \int_0^T (
      \|\tilde{U}^{n_k}\|^2 + \|\tilde{U}^{n_k}\|^4 + 1)ds.
  \end{split}
\end{displaymath}
Thus, for every $q \in [1,2)$,
\begin{displaymath}
  \begin{split}
    &\left\{
       \int_0^t 
       \langle
       A\tilde{U}^{n_k}+
       \theta( \|\tilde{U}^{n_k}- \tilde{U}^{n_k}_* \|)
         B^{n_k}(\tilde{U}^{n_k})+
       F^{n_k}(\tilde{U}^{n_k}), U^\sharp
       \rangle ds \right\}_{k \geq 0}\\
    &\quad \textrm{ is uniformly integrable in }
    L^q(\tilde{\Omega} \times [0,T]).
  \end{split}
\end{displaymath}
Combining this with (\ref{eq:FinalAConv}), (\ref{eq:FinalBconvConc}),
(\ref{eq:finalfinalfinalFpntWise}) we conclude that for every
$q \in [1,2)$:
\begin{equation}\label{eq:FinalConConvForDeterministicTerms}
  \begin{split}
    \int_0^t 
       &\langle
       A\tilde{U}^{n_k}+ 
       \theta( \|\tilde{U}^{n_k}- \tilde{U}^{n_k}_* \|)
        B^{n_k}(\tilde{U}^{n_k})+
      F^{n_k}(\tilde{U}^{n_k}), U^\sharp \rangle ds\\
       &\longrightarrow
    \int_0^t 
       \langle
       A\tilde{U} + 
       \theta( \|\tilde{U} - \tilde{U}_* \|)
         B(\tilde{U})+
         F(\tilde{U}), U^\sharp \rangle ds,\\
  \end{split}
\end{equation}
in $L^q([0,T] \times \Omega)$.

The stochastic terms are handled differently.  Using
(\ref{eq:decompEstimates}) and (\ref{eq:BndCondSig}) we estimate
\begin{displaymath}
  \begin{split}
   \|\sigma^{n_k}(\tilde{U}^{n_k})
        -\sigma(\tilde{U})\|_{L_2(\mathfrak{U},H)}
        \leq&    
   \|\sigma(\tilde{U}^{n_k})
        -\sigma(\tilde{U})\|_{L_2(\mathfrak{U},H)} +
     \|Q_{n_k}\sigma(\tilde{U})\|_{L_2(\mathfrak{U},H)}\\
     \leq & 
   \|\sigma(\tilde{U}^{n_k})
        -\sigma(\tilde{U})\|_{L_2(\mathfrak{U},H)} +
     \frac{1}{\lambda_{n_k}^{1/2}}\|\sigma(\tilde{U})\|_{L_2(\mathfrak{U},V)}\\
     \leq&
   \|\sigma(\tilde{U}^{n_k})
        -\sigma(\tilde{U})\|_{L_2(\mathfrak{U},H)} +
     \frac{c}{\lambda_{n_k}^{1/2}}(1 + \|\tilde{U} \|).\\
  \end{split}
\end{displaymath}
Thus, due to (\ref{eq:pointwiseTimeConv}) and the assumed continuity of $\sigma$
(see (\ref{eq:BndCondSig})) we conclude that
\begin{displaymath}
  \|\sigma^{n_k}(\tilde{U}^{n_k})
        -\sigma(\tilde{U})\|_{L_2(\mathfrak{U},H)} \rightarrow 0,
\end{displaymath}
for almost every $(\omega, t) \in \tilde{\Omega} \times [0,T]$.   On
the other hand, we observe that
\begin{displaymath}
  \sup_{n_k}\E \left( \int_0^T \|
  \sigma^{n_k}(\tilde{U}^{n_k})\|_{L_2(\mathfrak{U}, H)}^4 ds \right)
  \leq   c \sup_{n_k} \E \left( \int_0^T (1 + \|
  \tilde{U}^{n_k}\|^4) ds \right),
\end{displaymath}
where we have again made use of the sublinear condition
(\ref{eq:BndCondSig}).   We therefore
infer that $\|\sigma^{n_k}(\tilde{U}^{n_k})\|_{L_{2}(\mathfrak{U}, H)}$ 
is uniformly integrable in $L^p(\Omega \times [0,T] )$ for any $p \in [1, 4)$.
With the Vitali convergence theorem we infer, for all such $p \in [1, 4)$,
\begin{equation}\label{eq:suffCondLemma21}
  \sigma^{n_k}(\tilde{U}^{n_k}) 
  \rightarrow  \sigma(\tilde{U})
  \quad \textrm{ in } L^p(\tilde{\Omega};L^p([0,T], L_2(\mathfrak{U},H))).
\end{equation}
In particular \eqref{eq:suffCondLemma21} implies the convergence
in probability of $\sigma^{n_k}(\tilde{U}^{n_k})$ in $L^2([0,T],$ $L_2(\mathfrak{U},H)))$.
Thus, along with the assumption \eqref{eq:meaningoftopW},
we apply Lemma~\ref{eq:StochasticVitaliConc} and infer that
\begin{equation}\label{eq:lemma21conlusions}
  \int_0^t 
     \sigma^{n_k}(\tilde{U}^{n_k})
  d\tilde{W}_{n_k}
        \rightarrow   
    \int_0^t
   \sigma(\tilde{U})
    d\tilde{W},
\end{equation}
in probability $L^2([0,T], H))$.   Another application of the Vitali convergence 
theorem using estimates involving \eqref{eq:BDG}, \eqref{eq:suffCondLemma21} 
shows that the convergence in \eqref{eq:lemma21conlusions} occurs moreover in
$L^{2}(\Omega; L^2([0,T], H))$.

With the above details in hand we now establish \eqref{eq:modFullSystemCut}
in a variational sense.
Fix any $U^\sharp \in D(A)$, $\mathcal{R} \subset \tilde{\Omega}
\times [0,T]$ measurable.  Using \eqref{eq:weakConvDA1ActualLim} 
and then \eqref{eq:FinalConConvForDeterministicTerms} 
and \eqref{eq:lemma21conlusions} we observe that

\begin{displaymath}
  \begin{split}
    \E &\int_0^T \chi_{\mathcal{R}} \langle \tilde{U}, U^\sharp\rangle dt
   =
    \lim_{k \rightarrow \infty}\E 
     \int_0^T \chi_{\mathcal{R}} \langle \tilde{U}^{n_k}, U^\sharp\rangle dt\\
   =& 
   \lim_{k \rightarrow \infty}\E 
     \int_0^T \chi_{\mathcal{R}} \langle \tilde{U}_0^{n_k}, U^\sharp\rangle dt\\
     &-
     \lim_{k \rightarrow \infty}\E 
       \int_0^T \chi_{\mathcal{R}} 
       \left(\int_0^t \langle 
         A\tilde{U}^{n_k}
        ,  U^\sharp\rangle  ds
      \right) dt\\
     &-
     \lim_{k \rightarrow \infty}\E 
      \int_0^T \chi_{\mathcal{R}} 
      \left( 
        \int_0^t \langle 
        \theta( \|\tilde{U}^{n_k}- \tilde{U}^{n_k}_* \|)
        B^{n_k}(\tilde{U}^{n_k})
       ,  U^\sharp\rangle  ds
     \right) dt\\
     &-
     \lim_{k \rightarrow \infty}\E 
     \int_0^T \chi_{\mathcal{R}} 
     \left( 
       \int_0^t \langle 
      F^{n_k}(\tilde{U}^{n_k})
       ,  U^\sharp\rangle  ds
    \right) dt\\
     &+
     \lim_{k \rightarrow \infty}\E 
     \int_0^T \chi_{\mathcal{R}} 
     \left( 
       \int_0^t \langle 
      \sigma^{n_k}(\tilde{U}^{n_k})
       ,  U^\sharp\rangle  dW^{n_k}
    \right) dt\\
  =& 
   \E 
      \int_0^T \chi_{\mathcal{R}} \left( 
        \langle \tilde{U}_0, U^\sharp \rangle -
      \int_0^t \langle 
        A\tilde{U}+
        \theta( \|\tilde{U}- \tilde{U}_* \|)
      B(\tilde{U}) +
      F(\tilde{U})
       ,  U^\sharp\rangle  ds
     \right) dt\\
   &+
  \E 
    \int_0^T \chi_{\mathcal{R}} 
    \left( 
      \int_0^t \langle 
    \sigma(\tilde{U})
     ,  U^\sharp\rangle  dW
   \right) dt.
\end{split}
\end{displaymath}
Since this equality holds over all such $\mathcal{R}$ we may
conclude that for almost every $(\omega, t) \in \tilde{\Omega} \times
[0,T]$  and every $U^{\sharp} \in D(A)$ that,
\begin{equation}\label{eq:conclusionsOnVariationalEquality}
  \begin{split}
 \langle \tilde{U}(t), U^\sharp \rangle  +
  \int_0^t \langle 
        A\tilde{U}+
        \theta( \|\tilde{U}- \tilde{U}_* \|)&
      B(\tilde{U}) +
      F(\tilde{U})
       ,  U^\sharp\rangle  ds
  =   \langle \tilde{U}_0, U^\sharp \rangle +
    \int_0^t \langle 
    \sigma(\tilde{U})
     ,  U^\sharp\rangle  dW.
  \end{split}
\end{equation}
Moreover, due to \eqref{eq:RegConclusionsPart1} established above,
it follows by density that \eqref{eq:conclusionsOnVariationalEquality}
holds also over $U^{\sharp} \in H$ and hence \eqref{eq:modFullSystemCut}
in the analogous sense to \eqref{eq:spdeAbstracMG}.

\subsection{Improved Regularity In Time}
\label{sec:impr-regul-time}

With \eqref{eq:conclusionsOnVariationalEquality} and \eqref{eq:RegConclusionsPart1} 
in hand it remains only
to establish better continuity, in time, for $U$.  More precisely, we 
must show that 
$\tilde{U} \in C([0,T]; V)$ a.s.  Of course, such a condition is needed
in order to justify the definition \eqref{eq:firstExistenceTime}.

To this end we define
\begin{equation}\label{eq:AuxLinearSystem}
  d Z + AZ = \sigma(\tilde{U}) d\tilde{W}, \quad Z(0) = \tilde{U}_0.
\end{equation}
Observe that since
$\sigma(\tilde{U}) \in L^2(\Omega, L^2([0,T], L_2(\mathfrak{U}, V)))$
we have
\begin{equation}\label{eq:ZedReg}
Z \in L^2(\tilde{\Omega}, C([0,T], V)) \cap L^2(\tilde{\Omega}, L^2([0,T]; D(A))).
\end{equation}
Now take $\bar{U} = \tilde{U} - Z$.  Subtracting
\eqref{eq:AuxLinearSystem}  from \eqref{eq:modFullSystemCut}
we find
\begin{equation}\label{eq:equationForUbar}
  \begin{split}
  \frac{d}{dt} \bar{U} + A \bar{U} 
        + \theta(\| \bar{U}  &+ Z - \tilde{U}_* \|) B (\bar{U} + Z)
        + F(\bar{U} + Z)  = 0, \\
        \bar{U}(0) &= \tilde{U}_0.
 \end{split}
\end{equation}
Due to \eqref{eq:ZedReg}, \eqref{eq:RegConclusionsPart1},
we infer that $\bar{U} \in L^2(\tilde{\Omega}, L^2(0,T; D(A)) \cap L^{\infty}([0,T], V)$
and hence that,
\begin{equation}\label{eq:RegularityRHSUbarEqn}
    A \bar{U}, \;\;
    \theta(\| \bar{U}  + Z - \tilde{U}_* \|) B (\bar{U} + Z), \;\;
    F(\bar{U} + Z)
    \in L^2(\tilde{\Omega}, L^2([0,T], H).
\end{equation}
We conclude with \eqref{eq:equationForUbar} that
\begin{equation}\label{eq:SetUpForTemamTimeRegThm}
  \frac{d}{dt} A^{1/2} \bar{U} \in L^{2}(\Omega; L^2(0,T;V')), \quad
  A^{1/2} \bar{U} \in L^{2}(\Omega; L^{2}(0,T;V)).
\end{equation}
By applying
\cite[Chapter 3, Lemma 1.2]{Temam1} we infer that $A^{1/2} \bar{U} \in C([0,T], H)$ so
that, with \eqref{eq:ZedReg}, we deduce that
\begin{equation}\label{eq:almostsureContinuityConclusionBarU}
  \tilde{U} \in C([0,T];V), \; a.s.
\end{equation}
With \eqref{eq:conclusionsOnVariationalEquality}, \eqref{eq:RegConclusionsPart1},
and \eqref{eq:almostsureContinuityConclusionBarU} we finally conclude that 
$(\tilde{\mathcal{S}}, \tilde{U})$ is a global Martingale solution of \eqref{eq:modFullSystemCut}.
Furthermore, having justified \eqref{eq:firstExistenceTime} and applying 
\eqref{eq:LookMaNoCutoffLocally} to \eqref{eq:conclusionsOnVariationalEquality}
we have that $(\tilde{\mathcal{S}}, \tilde{U}, \tau)$ is a local Martingale solution of \eqref{eq:AbsEq}.
The proof of Proposition~\ref{thm:PassageToLimitThm} is therefore complete.

\section{Appendix: Proof of the Convergence Theorem}
\label{sec:append-proof-conv}

In this final section we provide a proof of Lemma~\ref{thm:ConvThm}. 
Convergence results similar to Lemma~\ref{thm:ConvThm} have appeared in previous works 
(see e.g. \cite{Bensoussan1}, \cite{GyongyKrylov1}). However, to the best of our knowledge, 
no one up to the present has provided a detailed proof.
Note that in the present work Lemma~\ref{thm:ConvThm} is an important technical tool for the 
passage to the limit, as
detailed above in Section~\ref{sec:passToLim}.    

To simplify the exposition, we begin by introducing the notations:
\begin{displaymath}
 \begin{split}
    \mathcal{I}^n := \int_0^t G^n dW^n = \sum_{k\geq 0} \int_0^t G^n_kdW_k^n 
    = \sum_{k \geq 0} Y_k^n, \quad \quad
    \mathcal{I} :=\int_0^t G dW = \sum_k \int_0^t G_kdW_k = \sum_k Y_k.
  \end{split}
\end{displaymath}
For the truncations we set
\begin{displaymath}
 \begin{split}
    \mathcal{I}_N^n :=  \sum_{ N \geq k \geq 0} Y_k^n,  \quad
    \mathcal{J}_{N}^{n} := \mathcal{I}^{n} -    \mathcal{I}_N^n, \quad
    \mathcal{I}_N :=  \sum_{ N \geq k \geq 0} Y_k, \quad
    \mathcal{J}_N := \mathcal{I}^{n} -    \mathcal{I}_N^n.
  \end{split}
\end{displaymath}
With these notations we now split 
\begin{displaymath}
\begin{split}
   |\mathcal{I}^n& - \mathcal{I}|_{L^2([0,T],X)}
    \leq |\mathcal{I}^n - \mathcal{I}^n_N|_{L^2([0,T],X)}+
    |\mathcal{I}^n_N - \mathcal{I}_N|_{L^2([0,T],X)}
    +
    |\mathcal{I}_N - \mathcal{I}|_{L^2([0,T],X)}\\
  \end{split}
\end{displaymath}
and observe that the proof of Lemma~\ref{thm:ConvThm} is complete once we 
establish that
\begin{equation}\label{eq:SufficentCond3TermsConvThm}
\begin{cases}
\textrm{For every } \epsilon > 0, 
\lim\limits_{N \rightarrow \infty} \sup\limits_{n \geq N} \Prb \left(
     |\mathcal{J}^n_N|_{L^2([0,T],X)} > \epsilon
          \right) =0,\\
\lim\limits_{n \rightarrow \infty} |Y^{n}_{k} - Y_{k}|_{L^2([0,T],X)} = 0 \textrm{ in Probability, for each fixed } k,\\
\lim\limits_{N \rightarrow \infty} |\mathcal{J}_N |_{L^2([0,T],X)} = 0 \textrm{ in Probability}.
\end{cases}
\end{equation}

To establish each of the convergences in \eqref{eq:SufficentCond3TermsConvThm} we 
make extensive use of the following martingale inequality (see e.g. \cite{GihmanSkorohod1})
\begin{equation}\label{eq:GilSkInEq}
\begin{split}
    \Prb \left( \int_{0}^{T} \left| \int_{0}^{t} F dW \right|^{2}_{X} dt  > c \right)
    \leq \frac{\kappa T}{c} + \Prb \left( \int_{0}^{T} |F|^{2}_{L_{2}(\mathfrak{U},X)} dt > \kappa \right).
\end{split}
\end{equation}
Here $c$, $\kappa$ may be any positive constants and $F$ any $\mathcal{F}_{t}$ predictable
element in $L^{2}([0,T]; L_{2}(\mathfrak{U}, X))$.  For the first item in \eqref{eq:SufficentCond3TermsConvThm}
we apply \eqref{eq:GilSkInEq} and observe that for any $\epsilon, \delta >0$
\begin{displaymath}
\begin{split}
   \Prb \left(|\mathcal{J}^n_{N} |_{L^2([0,T],X)} > \epsilon
          \right) 
          &\leq \frac{\delta}{3} +    
          \Prb \left(  \sum_{k \geq N} \int_{0}^{T}| G^{n}_{k} |^{2} dt >  \frac{\delta \epsilon^{2}}{3 T} \right)\\
          &\leq \frac{\delta}{3} +    
          \Prb \left(  \int_{0}^{T} \! \! |G^{n} - G|^{2}_{L_{2}(\mathfrak{U}, X)}  dt >  \frac{\delta \epsilon^{2}}{12 T} \right)
          + \Prb \left(   \sum_{k \geq N}  \int_{0}^{T}  \! \!  | G|^{2} dt >  \frac{\delta \epsilon^{2}}{12 T} \right).
\end{split}
\end{displaymath}
With this estimate, the assumptions on $G$ and \eqref{eq:ConvAsmpStochIntegrand} we infer
the first item in \eqref{eq:SufficentCond3TermsConvThm}.  The third item in 
\eqref{eq:SufficentCond3TermsConvThm} is established in similar manner via an application of
\eqref{eq:GilSkInEq}.

It remains to address the second item in \eqref{eq:SufficentCond3TermsConvThm}.
In order to treat these terms we introduce the functional:
\begin{equation}\label{eq:smoothingFunctional}
  \mathcal{R}_{\rho}(F) 
  = \frac{1}{\rho} \int_0^t 
  \exp\left( - \frac{t -s}{\rho} \right) F(s) ds
  \quad F \in L^{1}([0,T], X), \rho > 0.
\end{equation}
Using this functional and then integrating by parts we estimate
\begin{equation}\label{eq:FirstSplittingEsp}
  \begin{split}
  |Y_{k}^{n}& - Y_{k}|_{X} = \biggl| \int_0^t G^n_k dW^n_k - \int_0^t G_k dW_k  \biggr|_{X}\\
  \leq&
 \left| \int_0^t ( G^n_k - 
   \mathcal{R}_\rho(G^n_{k}) )dW^n_k\right|_X +
 \left| \int_0^t (
  \mathcal{R}_\rho(G_{k}) - G_k) dW_k \right|_X
  +
 \left|\int_0^t \mathcal{R}_\rho(G^n_{k}) dW^n_k
   -\int_0^t \mathcal{R}_\rho(G_{k}) dW_k \right|^2_X\\
 \leq&
  \left| \int_0^t ( G^n_k - 
   \mathcal{R}_\rho(G^n_{k}) )dW^n_k\right|_X +
 \left| \int_0^t (
  \mathcal{R}_\rho(G_{k}) - G_k) dW_k \right|_X 
  + |\mathcal{R}_\rho(G_k) W_k -\mathcal{R}_\rho(G^n_{k}) W^n_k|_X \\
  & + \frac{1}{\rho} \left|\int_0^t(\mathcal{R}_\rho(G_{k}) W_k
        -\mathcal{R}_\rho(G^n_{k}) W^n_k)ds\right|_X
  + \frac{1}{\rho} \left|\int_0^t( G_k  W_k
        -G^n_k W^n_k)ds \right|_X. \\
  \end{split}
\end{equation}

We now proceed to treat each of the term on the right hand side of \eqref{eq:FirstSplittingEsp}.  
Fix $\epsilon, \delta > 0$. For the first term in \eqref{eq:FirstSplittingEsp} we apply \eqref{eq:GilSkInEq} and estimate
\begin{equation}\label{eq:conStTermJ1}
\begin{split}
 \Prb\biggl( \biggl| \int_0^t ( G^n_k& - 
   \mathcal{R}_\rho(G^n_{k}) )dW^n_k \biggr|_{L^{2}([0,T];X)} > \epsilon \biggr)
    \leq \delta 
         + \Prb \left(  \int_{0}^{T} |G^{n}_{k} - \mathcal{R}_{\rho} ( G^{n}_{k}) |^{2}_{X} dt > \frac{\delta \epsilon^{2}}{T} \right)\\
     \leq& \delta  +
       \Prb\left(  \int_0^T |G_k^n - G_k|^2_X dt > \frac{\delta \epsilon^{2}}{3 T} \right) 
       + \Prb\left(  \int_0^T |G_k - \mathcal{R}_\rho(G_{k})|^2_X dt > \frac{\delta \epsilon^{2}}{3 T} \right) \\
       &+ \Prb\left(  \int_0^T | \mathcal{R}_\rho(G_{k}) - \mathcal{R}_\rho(G^n_{k})|^2_X dt > \frac{\delta \epsilon^{2}}{3 T} \right)\\
    \leq& \delta  + 2
       \Prb\left(  \int_0^T |G_k^n - G_k|^2_X dt > \frac{\delta \epsilon^{2}}{3 T} \right) 
       + \Prb\left(  \int_0^T |G_k - \mathcal{R}_\rho(G_{k})|^2_X dt > \frac{\delta \epsilon^{2}}{3 T} \right). \\
\end{split}
\end{equation}
With \eqref{eq:GilSkInEq} we also find that
\begin{equation}\label{eq:conStTermJ2}
\begin{split}
 \Prb \biggl( \biggl|\int_0^t (
  \mathcal{R}_\rho(G_{k}) - G_k) dW_k \biggr|_{L^{2}([0,T];X)} > \epsilon \biggr)
    \leq \delta  + \Prb\left(  \int_0^T |G_k - \mathcal{R}_\rho(G_{k})|^2_X dt > \frac{\delta \epsilon^{2}}{T} \right)
\end{split}
\end{equation}
The last three items are treated differently
\begin{equation}\label{eq:conStTermJ3}
\begin{split}
 \Prb \biggl(& |\mathcal{R}_\rho(G_k) W_k -\mathcal{R}_\rho(G^n_{k}) W^n_k|_{L^{2}([0,T];X)} > \epsilon \biggr)\\
    \leq& \Prb \left( \int_0^T |\mathcal{R}_\rho(G_k) W_k^{n}
            -\mathcal{R}_\rho(G^n_k) W_k^{n}|^2_X dt   > \frac{\epsilon^{2}}{4} \right)
         + \Prb 
         	        \left( \int_0^T 
	        |\mathcal{R}_\rho(G_k) W_k
            -\mathcal{R}_\rho(G_k) W^n_k|^2_X dt >   \frac{\epsilon^{2}}{4} \right)\\
      \leq& \Prb \left( \sup_{t \in [0,T]} |W_{k}^{n}|^{2} \int_0^T |G_k
            -G^n_k |^2_X dt   > \frac{\epsilon^{2}}{4} \right)
         + \Prb 
         	        \left(  \sup_{t \in [0,T]} |W_{k} -W_{k}^{n}|^{2}  \int_0^T 
	        |G_{k}|^2_X dt >   \frac{\epsilon^{2}}{4} \right).\\
\end{split}
\end{equation}
Similar estimates lead to
\begin{equation}\label{eq:conStTermJ4}
\begin{split}
 \Prb \biggl( \biggl| \frac{1}{\rho} &\int_0^t(\mathcal{R}_\rho(G_{k}) W_k
        -\mathcal{R}_\rho(G^n_{k}) W^n_k)ds \biggr|_{L^{2}([0,T];X)}
         > \epsilon \biggr)\\
      \leq& \Prb \left( \sup_{t \in [0,T]} |W_{k}^{n}|^{2} \int_0^T |G_k
            -G^n_k |^2_X dt   > \frac{\epsilon^{2} \rho^{2}}{4T^{2}} \right)
         + \Prb 
         	        \left(  \sup_{t \in [0,T]} |W_{k} -W_{k}^{n}|^{2}  \int_0^T 
	        |G_{k}|^2_X dt >   \frac{\epsilon^{2} \rho^{2}}{4T^{2}} \right).\\
\end{split}
\end{equation}
The final term in \eqref{eq:FirstSplittingEsp} yields to an identical estimate.

Collecting the estimates \eqref{eq:conStTermJ1}, \eqref{eq:conStTermJ2}, \eqref{eq:conStTermJ3}, 
\eqref{eq:conStTermJ4} we infer that
\begin{equation}\label{eq:conStTermSummary}
\begin{split}
   \Prb (& |Y_{k}^{n} - Y_{k}|_{L^{2}([0,T], X)}  > 5\epsilon)\\
       \leq& 2\delta +
       2\Prb\left(  \int_0^T |G_k^n - G_k|^2_X dt > \frac{\delta \epsilon^{2}}{3 T} \right) 
       + 2\Prb\left(  \int_0^T |G_k - \mathcal{R}_\rho(G_{k})|^2_X dt > \frac{\delta \epsilon^{2}}{3 T} \right) \\
       &+ 3\Prb \left( \sup_{t \in [0,T]} |W_{k}^{n}|^{2} \int_0^T |G_k
            -G^n_k |^2_X dt   > \frac{\epsilon^{2} \rho^{2}}{4T} \right)
         + 3\Prb 
         	        \left(  \sup_{t \in [0,T]} |W_{k} -W_{k}^{n}|^{2}  \int_0^T 
	        |G_{k}|^2_X dt >   \frac{\epsilon^{2} \rho^{2}}{4T} \right).\\
\end{split}
\end{equation}
Since $\delta, \epsilon > 0$ are arbitrary and given basic properties of the 
functional \eqref{eq:smoothingFunctional} along with \eqref{eq:ConvAsmp}
we may now infer the second item of \eqref{eq:SufficentCond3TermsConvThm}
from \eqref{eq:conStTermSummary}.  This completes the proof of Lemma~~\ref{thm:ConvThm}.

\section*{Acknowledgments}
This work was partially supported by the National Science Foundation under the grants
NSF grants DMS-1004638, DMS-0906440, and by the Research Fund of Indiana University.

\footnotesize
\bibliographystyle{plain}
\bibliography{/Users/Nathan/Desktop/Work/Research/Reference/ref}
\normalsize

\noindent Arnaud Debussche\\ {\footnotesize
D\'{e}partement de Math\'{e}matiques \\
ENS Cachan Bretagne\\
Web: \url{http://w3.bretagne.ens-cachan.fr/math/people/arnaud.debussche/}\\
Email: \url{arnaud.debussche@bretagne.ens-cachan.fr} } \\[.3cm]
Nathan Glatt-Holtz\\ {\footnotesize
Department of Mathematics\\
Indiana University\\
Web: \url{http://mypage.iu.edu/\~negh/}\\
 Email: \url{negh@indiana.edu}} \\[.3cm]
Roger Temam\\ {\footnotesize
Department of Mathematics\\
Indiana University\\
Web: \url{http://mypage.iu.edu/~temam/}\\
 Email: \url{temam@indiana.edu}}

\end{document}